\documentclass[11pt,a4paper,twoside]{article}

\usepackage{setspace}
\usepackage{fancyhdr}
\usepackage[utf8]{inputenc}
\usepackage[T1]{fontenc}
\usepackage{fixltx2e}
\usepackage{graphicx}
\usepackage{longtable}
\usepackage{float}
\usepackage{wrapfig}
\usepackage{soul}
\usepackage{textcomp}
\usepackage{marvosym}
\usepackage[nointegrals]{wasysym}
\usepackage{latexsym}
\usepackage{amssymb}
\usepackage{hyperref}
\usepackage{amsmath,amsthm}
\usepackage{mathrsfs}
\usepackage[autolanguage]{numprint}
\usepackage{dcolumn}
\newcolumntype{d}[1]{D{.}{.}{#1}}
\usepackage{algorithmicx}
\usepackage{algorithm}
\usepackage{algpseudocode}
\usepackage[width=14.5cm]{geometry}

\DeclareMathOperator{\mean}{E}
\DeclareMathOperator{\prob}{Pr}
\DeclareMathOperator{\divkl}{KL}
\DeclareMathOperator{\lpl}{LPL}

\DeclareMathOperator{\logit}{logit}

\DeclareMathOperator*{\amax}{arg\,max}
\newcommand{\R}{\mathbb{R}}
\newcommand{\sprod}[2]{\left\langle #1,#2\right\rangle}
\newcommand{\nseint}[1]{\overset{\circ}{n}_\text{s}(#1)}
\newcommand{\nbseint}[1]{\overset{\circ}{n}_{\text{s,b}}\left(#1\right)}
\newcommand{\nnbseint}[1]{\overset{\circ}{n}_{\text{s,nb}}\left(#1\right)}
\newcommand{\classname}[1]{\textsf{#1}}
\newcommand{\funname}[1]{\textsf{#1}}

\newtheorem{lemma}{Lemma}
\newtheorem{remark}{Remark}
\newtheorem{condition}{Condition}
\newtheorem{corollary}{Corollary}
\newtheorem{proposition}{Proposition}
\newtheorem{definition}{Definition}
\theoremstyle{definition}
\newtheorem{example}{Example}
\newcommand{\colhead}[1]{\multicolumn{1}{c}{#1}}
\DeclareMathOperator*{\argmax}{arg\,max}

\newcommand{\HRule}{\rule{\linewidth}{0.5mm}}

\newlength{\titleskip}
\graphicspath{{./}{figures/}}

\begin{document}

\begin{titlepage}

\begin{center}      
\HRule \\[0.3cm] 
  \begin{spacing}{1.5}
         {\Large
          \textbf{         
          Pseudolikelihood inference for Gibbsian T-tessellations\ldots{} and point processes         
         }}
  \end{spacing}  
\HRule
\end{center}

\vspace{\titleskip}

\begin{center}
  \begin{spacing}{1.5}
        {\large 
         \textbf{
         Kiên \textsc{Kiêu} and Katarzyna \textsc{Adamczyk-Chauvat}
        }}
  \end{spacing}  
\end{center}

\vfill 
\noindent 
\textbf{Rapport technique 2015-2, novembre 2015
\\[2\baselineskip]
UR1404 Math\'ematiques et Informatique Appliqu\'ees du Génome à l'Environnement\\
INRA\\
Domaine de Vilvert\\
78352 Jouy-en-Josas Cedex\\
France\\
\url{http://maiage.jouy.inra.fr}}
\\[2\baselineskip]
{\small \copyright \, 2015 INRA
}
\begin{figure}[hb]
\includegraphics[width=50mm]{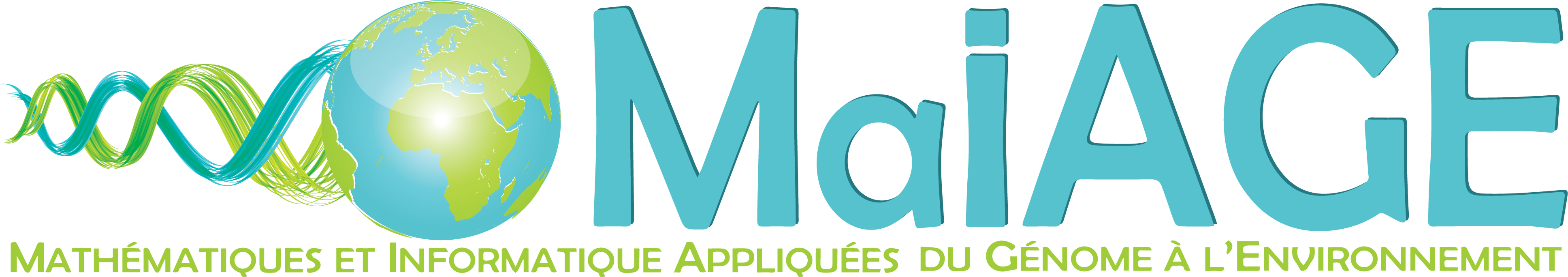}\hfill\includegraphics[width=30mm]{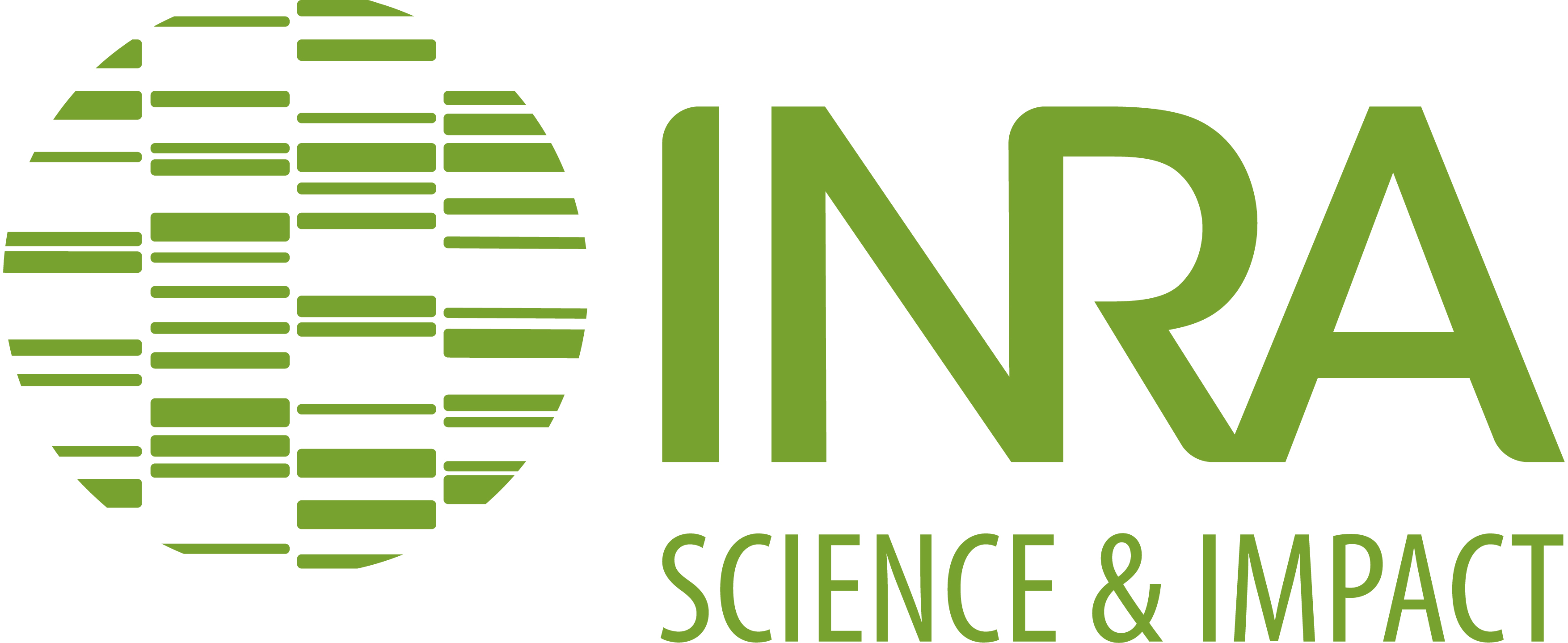}
\end{figure}
\end{titlepage}

\thispagestyle{plain}

\noindent 
Kiên \textsc{Kiêu}\\
UR1404 MaIAGE, INRA, domaine de Vilvert, 78352 Jouy-en-Josas Cedex\\
kien.kieu@jouy.inra.fr\\[\baselineskip]
Katarzyna \textsc{Adamczyk-Chauvat}\\
UR1404 MaIAGE, INRA, domaine de Vilvert, 78352 Jouy-en-Josas Cedex

\newpage

%
\renewcommand{\abstractname}{Abstract}
\begin{abstract}
  Recently a new class of planar tessellations, named T-tessellations, was introduced. Splits, merges and a third local modification named flip where shown to be sufficient for exploring the space of T-tessellations. Based on these local transformations and by analogy with point process theory, tools Campbell measures and a general simulation algorithm of Metropolis-Hastings-Green type were translated for random T-tessellations.

  The current report is concerned with parametric inference for Gibbs models of T-tessellations. The estimation criterion referred to as the pseudolikelihood is derived from Campbell measures of random T-tessellations and the Kullback-Leibler divergence. A detailed algorithm for approximating the pseudolikelihood maximum is provided. A simulation study seems to show that bias and variability of the pseudolikelihood maximum decrease when the tessellated domain grows in size.

  In the last part of the report, it is shown that an analogous approach based on the Campbell measure and the KL divergence when applied to point processes leads to the well-known pseudo-likelihood introduced by Besag. More surprisingly, the binomial regression method recently proposed by Baddeley and his co-authors for computing the pseudolikelihood maximum can be derived using the same approach starting from a slight modification of the Campbell measure.
\end{abstract}

\newpage

\vspace*{4cm}
\setcounter{tocdepth}{3}
\tableofcontents

\cleardoublepage
\section{Introduction}
\label{sec-1}

Recently, a new class of planar tessellations was introduced \cite{ref/1242}: Gibbsian T-tessellations. Briefly, a T-tessellation is a tessellation with only vertices at the intersection of three edges, two of them being aligned. In \cite{ref/1242}, it was shown that the space of T-tessellations can be explored using three simple and local operators: splits, merges and flips. A model was proposed to be considered as a completely random T-tessellation model (CRTT). And Gibbsian variants were discussed. Also tools and results such as Campbell measures, Papangelou kernels, Georgii-Nguyen-Zessin formulae were translated from point process theory to T-tessellations. A Metropolis-Hastings-Green algorithm based on splits, merges and flips was derived for the simulation of Gibbsian T-tessellations. Such an algorithm defines a Split, Merge and Flip (SMF) Markov chain converging to the Gibbs distribution to be simulated. As a companion tool, a C++ library named LiTe \cite{liteweb} was developped. Using LiTe, it is possible to simulate a rather large class of T-tessellation Gibbs models. 

The main purpose of this report is to introduce a method for estimating the parameters of a T-tessellation Gibbs model. This inference method is based on a so-called pseudolikelihood. The approach developped here involves Campbell measures and the Kullback-Leibler divergence. It is shown that the well-known pseudolikelihood introduced by Besag \cite{ref/1243} for point process can also be built from the Campbell measure and the Kullback-Leibler divergence.

The proposal for the pseudolikelihood for Gibbsian T-tessellations is provided in Section \ref{sec:pseudo:tessellations}. Section \ref{sec:ttessel:exponential} focuses on the special case of a canonical exponential family. Practical computation of the pseudolikelihood maximum is discussed in Section \ref{sec:nois:algorithm}. The behaviour of the maximum pseudolikelihood estimator is investigated based on some numerical simulations in Section \ref{sec:simulations}. Section \ref{sec:wrong:model} is devoted to an unexpected drawback of the pseudolikelihood. 

In Section \ref{sec:contrast:pl:point:processes}, it is shown that a similar approach based on the Campbell measure and the Kullback-Leibler divergence provides a new way to obtain the widely used pseudolikelihood for point processes. Even the recently logistic regression method for computing the pseudolikelihood estimator can be obtained in such a way as shown in Section \ref{sec:pp:logistic:regression}.

Appendix \ref{sec:kl} is a technical prerequisite where the Kullback-Leibler divergence is extended to non-probability measures. Appendix \ref{sec:subconfig:poisson} provides a result about the mean of a functional summed over all subpatterns of a finite Poisson point process. Such a result is used for proving that the distribution of a Gibbsian T-tessellation is fully characterized by its Papangelou conditional intensities (detailed proof provided in Appendix \ref{sec:charac:papangelou}).
\section{Pseudolikelihood for T-tessellations}
\label{sec-2}
\subsection{Gibbsian T-tesselllations}
\label{sec-2-1}

A general theory of planar T-tessellations was sketched in \cite{ref/1242}. At the time being, the theory is limited to bounded domains. Let $D\subset\mathbb{R}^2$ be a convex bounded polygon. Below, $D$ is referred to as the domain. The space of T-tessellations of $D$ is denoted by $\mathcal{T}$. A T-tessellation has only convex cells and can also be seen as a collection of segments. A segment is a maximal union of aligned and connected edges. There are two classes of segments: non-blocking (single edge) and blocking (multiple edges) segments. Three local operators allow to explore $\mathcal{T}$:
\begin{description}
\item[Split] a cell is divided by a new segment.
\item[Merge] a non-blocking segment is removed.
\item[Flip] an edge at the end of a blocking segment is removed while the incident blocked segment is extended.
\end{description}
Note that the inverse of a split is a merge (and vice versa), while the inverse of a flip is a flip. Below $\mathbb{S}_T$, $\mathbb{M}_T$ and $\mathbb{F}_T$ refer to the spaces of splits, merges and flips that can be applied to tessellation $T$. Note that $\mathbb{S}_T$ is continuous while $\mathbb{M}_T$ and $\mathbb{F}_T$ are finite. In \cite{ref/1242}, uniform measures with infinitesimal elements denoted by $ds$, $dm$ and $df$ are defined. The last two measures are just counting measures. The measure $ds$ has a total mass equal to the sum of the perimeters of $T$ cells, denoted by $u(T)$, divided by $\pi$. The probability measure $\pi ds/u(T)$ on $\mathbb{S}_T$ is characterized by the two following facts:
\begin{itemize}
\item Cells of $T$ are split with probabilities proportional to their perimeters.
\item Given a cell $C$ of $T$ is split, the splitting line is a uniform and isotropic random line hitting $C$.
\end{itemize}

The so-called CRTT model introduced in \cite{ref/1242} plays the role of the Poisson point process in point process theory. Its distribution is denoted by $\mu$. 
\begin{definition}\label{def:crtt}
  For a given non-negative real number $\nu$, the CRTT on $D$ with
  parameter $\nu$ is defined by the following distribution on
  $\mathcal{T}$:
  \begin{equation}
    \label{eq:def:crtt}
    \mu_\nu(dT) \propto I\{T\in \mathcal{T}(L)\} \frac{\nu^n}{n!}\;dl_1\ldots dl_n, \quad \forall L=\{l_1,\ldots,l_n\},
  \end{equation}
  where $\mathcal{T}(L)$ is the subset of T-tessellations with segments supported by the line pattern $L$.
\end{definition}
For sake of brevity, we write $\mu$ the distribution of the CRTT with parameter $\nu=1$. CRTT stands for Completely Random T-Tessellation.

All other models considered below have a density with respect to $\mu$. Their distributions can be written as
\begin{equation}
  \label{eq:density:wrt:mu}
  P(dT) = h(T)\;\mu(dT).
\end{equation}
In general, $h$ is known analytically only up to a multiplicative normalizing constant. Even for the CRTT where $h$ is constant, the normalizing constant is not known at the moment. This is why inference by maximum likelihood is not straightforward.

The density $h$ is said to be hereditary in the sense defined in \cite{ref/1242} if:
\begin{itemize}
\item For every pair $(T,s)$, $T\in\mathcal{T}$, $s\in\mathbb{S}_T$, $h(T)=0\Rightarrow h(sT)=0$.
\item For every pair $(T,f)$, $T\in\mathcal{T}$, $f\in\mathbb{F}_T$, $h(T)=0\Rightarrow h(fT)=0$.
\end{itemize}
The condition was erroneoulsy stated in \cite{ref/1242} with inversed implications.

In \cite{ref/1242}, two reduced Campbell measures for random T-tessellations were introduced.
The split Campbell measure of a random T-tessellation $\mathbf{T}$ maps any function $\phi:\mathscr{C}_\text{s}=\{(s,T):T\in\mathcal{T},s\in\mathbb{S}_T\}\rightarrow\R$ to
\begin{equation}
  \label{eq:split:campbell}
  C^!_{\text{s}}(\phi) = \mean
  \sum_{m\in\mathbb{M}_\mathbf{T}} \phi\left(m^{-1},m\mathbf{T}\right).
\end{equation}
The split Campbell measure is very similar to the Campbell measure of a point process. Indeed splitting and merging are analog to insertion and removal of points.

The flip Campbell measure is more specific to T-tessellations but its definition is formally very close to the definition of the split Campbell measure. The flip Campbell measure of $\mathbf{T}$ maps any function $\phi:\mathscr{C}_\text{f}=\{(f,T):T\in\mathcal{T},f\in\mathbb{F}_T\}\rightarrow\R$ to
\begin{equation}
  \label{eq:flip:campbell}
  C^!_{\text{f}}(\phi) = \mean
  \sum_{f\in\mathbb{F}_\mathbf{T}} \phi\left(f^{-1},f\mathbf{T}\right).
\end{equation}

It was shown in \cite{ref/1242} that if $\mathbf{T}\sim h\mu$ where $h$ is hereditary then
\begin{eqnarray}
  \label{eq:split:campbell:hereditary}
  C^!_{\text{s}}(ds,dT) & = & \lambda_{\text{s}}(s,T)\;dsP(dT),\\
  \label{eq:flip:campbell:hereditary}
  C^!_{\text{f}}(f,dT) & = & \lambda_{\text{f}}(f,T)\;dfP(dT),
\end{eqnarray}
where
\begin{eqnarray}
  \label{eq:split:papangelou}
  \lambda_{\text{s}}(s,T) & = & \frac{h(sT)}{h(T)},\\
  \label{eq:flip:papangelou}
  \lambda_{\text{f}}(f,T) & = & \frac{h(fT)}{h(T)},
\end{eqnarray}
for $P$-almost all $T\in\mathcal{T}$, almost all $s\in\mathbb{S}_T$ and all $f\in\mathbb{F}_T$. The densities $\lambda_{\text{s}}$ and $\lambda_{\text{f}}$ of the reduced Campbell measures are called the split and flip Papangelou conditional intensities of $\mathbf{T}$. Here we will need a third one: the merge Papangelou conditional intensity defined by
\begin{equation}
  \label{eq:def:merge:papangelou}
  \lambda_{\text{m}}(m,T)  =  \frac{h(mT)}{h(T)}.
\end{equation}
Note that
\begin{equation}
  \label{eq:rel:papangelou:split:merge}
  \lambda_{\text{m}}(m,T) = \lambda_{\text{s}}(m^{-1},mT)^{-1}.
\end{equation}
The split and flip Papangelou conditional intensities uniquely determine the density $h$ when it is hereditary. A proof of the result below is given in Appendix \ref{sec:charac:papangelou}.
\begin{proposition} \label{prop:sf:papangelou:determines:dist}
  Let $\mathbf{T}\sim P=h\mu$ where $h$ is supposed to be hereditary. The density $h$ is uniquely determined up to a $\mu$-null set by any pair $(\lambda_{\text{s}},\lambda_{\text{f}}) $ of split and flip Papangelou conditional intensities of $\mathbf{T}$.
\end{proposition}

As this report focuses on parametric inference, most often a parametric family of distributions $P_\theta$ on $\mathcal{T}$ will be considered. The issue adressed in this report is to estimate $\theta$ given a realization of $P_\theta$.

Since a general simulation algorithm of Metropolis-Hastings-Green is provided in \cite{ref/1242}, inference by Monte-Carlo Maximum Likelihood is possible, see e.g. \cite{ref/934}. But its practical implementation can be quite demanding. So easier and more automatized inference methods are of interest especially in contexts where estimation has to be repeatedly applied with different models or on large data sets.

The pseudolikelihood approach developped below is based on Campbell measures and Kullback-Leibler divergence. Using discretization as in \cite{ref/1243} for point processes is hardly concevable for tessellations.
\subsection{Pseudolikelihood definition and properties}
\label{sec-2-2}

\label{sec:pseudo:tessellations}
Consider a parametric family of distributions $P_\theta$ on $\mathcal{T}$, $\theta\in\Theta$. Every $P_\theta$ is supposed to be absolutely continuous with respect to $\mu$ and the densities of the $P_\theta$'s are denoted by $h_\theta$. The model is supposed to be identifiable:
\begin{equation}\label{eq:identifiability:tessel}
  \theta=\theta^* \Leftrightarrow h_\theta(T)=h_{\theta^*}(T)
  \text{ for }\mu\text{-almost all }T\in\mathcal{T}.
\end{equation}
The densities $h_\theta$ are supposed to be hereditary. The split, merge and flip Papangelou intensities are denoted $\lambda_{\text{s},\theta}$, $\lambda_{\text{m},\theta}$ and $\lambda_{\text{f},\theta}$. 

We are now ready for the pseudolikelihood definition.
\begin{definition}
  Given an observed T-tessellation $T$, the log-pseudolikelihood is defined as
\begin{multline}
  \lpl(\theta|T) = -
  \sum_{m\in\mathbb{M}_T} \log\lambda_{\text{m},\theta}\left(m,T\right) -
  \int_{\mathbb{S}_T} \lambda_{\text{s},\theta}\left(s,T\right)\;ds -\\
  \sum_{f\in\mathbb{F}_T} \log\lambda_{\text{f},\theta}\left(f,T\right) -
  \sum_{f\in\mathbb{F}_T} \lambda_{\text{f},\theta}\left(f,T\right).
  \label{eq:pseudo:tessel}
\end{multline}
  The maximum pseudolikelihood estimator is defined as
  \begin{equation}
    \label{eq:ttessel:max:pseudo}
    \hat{\theta} = \argmax_{\theta\in\Theta} \lpl(\theta|T).
  \end{equation}
\end{definition}

Similarly to the log-likelihood, the log-pseudolikelihood  can be derived as an empirical estimator of a contrast function. As a starting point, instead of $P_\theta$, consider the split and flip Campbell measures: 
\begin{eqnarray}
  \label{eq:split:campbell:parametric}
  C^!_{\text{s},\theta}(ds,dT) & = & \lambda_{\text{s},\theta}(s,T)\;dsP_\theta(dT),\\
  \label{eq:flip:campbell:parametric}
  C^!_{\text{f},\theta}(f,dT) & = & \lambda_{\text{f},\theta}(f,T)\;dfP_\theta(dT),
\end{eqnarray}
Compared to $P_\theta$, the Campbell measures have densities that do not involve unknown normalizing constants. It should be noticed however that the dominating measure above depends on $\theta$. This is why we introduce further measures closely related to the Campbell measures. For every pair $(\theta,\theta^*)$, consider the measures $\alpha_{\text{s},\theta,\theta^*}$ on $\mathscr{C}_{\text{s}}$ and $\alpha_{\text{f},\theta,\theta^*}$ on $\mathscr{C}_{\text{f}}$ defined by
\begin{eqnarray}
  \label{eq:def:alpha:split}
  \alpha_{\text{s},\theta,\theta^*}(ds,dT) & = &
  \lambda_{\text{s},\theta}(s,T)\;dsP_{\theta^*}(dT),\\
  \label{eq:def:alpha:flip}
  \alpha_{\text{f},\theta,\theta^*}(df,dT) & = &
  \lambda_{\text{f},\theta}(f,T)\;dfP_{\theta^*}(dT).
\end{eqnarray}
The measures $\alpha_{\text{s},\theta,\theta^*}$ and $\alpha_{\text{f},\theta,\theta^*}$ can be viewed as candidate Campbell measures. 
When $\theta=\theta^*$, they are indeed Campbell measures. 

For the time being, let us focus on the case where  $\alpha_{\text{s},\theta,\theta^*}\gg\alpha_{\text{s},\theta^*,\theta^*}$ and $\alpha_{\text{f},\theta,\theta^*}\gg\alpha_{\text{f},\theta^*,\theta^*}$, then
\begin{equation*}
  \frac{
    d\alpha_{\text{s},\theta^*,\theta^*}
  }{
    d\alpha_{\text{s},\theta,\theta^*}
  }(s,T) =
  \frac{\lambda_{\text{s},\theta^*}(s,T)}{\lambda_{\text{s},\theta}(s,T)},\quad
  \frac{
    d\alpha_{\text{f},\theta^*,\theta^*}
  }{
    d\alpha_{\text{f},\theta,\theta^*}
  }(f,T) =
  \frac{\lambda_{\text{f},\theta^*}(f,T)}{\lambda_{\text{f},\theta}(f,T)}.
\end{equation*}

It is standard in statistics to quantify the dissemblance between distributions using the Kullback-Leibler divergence.
Since the measures $\alpha_{\text{s},\theta,\theta^*}$ and $\alpha_{\text{f},\theta,\theta^*}$ are not probability measures, using the Kullback-Leibler divergence requires some adaptation. Appendix \ref{sec:kl} provides such an extension to non-negative non-probability measures. The extended Kullback-Leibler divergence of $\alpha_{\text{s},\theta,\theta^*}$ from $\alpha_{\text{s},\theta^*,\theta^*}$ is defined as
\begin{eqnarray*}
  \lefteqn{
    \divkl\left(\alpha_{\text{s},\theta^*,\theta^*},
    \alpha_{\text{s},\theta,\theta^*}\right)}\\
  & = &
  \int_{\mathscr{C}_\text{s}} \left(
    \frac{\lambda_{\text{s},\theta^*}(s,T)}{\lambda_{\text{s},\theta}(s,T)}
    \log\frac{\lambda_{\text{s},\theta^*}(s,T)}{\lambda_{\text{s},\theta}(s,T)} +
    1 - \frac{\lambda_{\text{s},\theta^*}(s,T)}{\lambda_{\text{s},\theta}(s,T)} \right)
  \;\alpha_{\text{s},\theta,\theta^*}(ds,dT)
\end{eqnarray*}
Note that for every $\theta\in\Theta$
\begin{equation*}
  \divkl\left(\alpha_{\text{s},\theta^*,\theta^*},
    \alpha_{\text{s},\theta,\theta^*}\right) \geq 0
\end{equation*}
and that the extended Kullback-Leibler divergence cancels for $\theta=\theta^*$.

Replacing in the right-hand side $\alpha_{\text{s},\theta,\theta^*}$ by its definition \eqref{eq:def:alpha:split}, one gets
\begin{eqnarray*}
  \lefteqn{
    \divkl\left(\alpha_{\text{s},\theta^*,\theta^*},
    \alpha_{\text{s},\theta,\theta^*}\right)}\\
  & = &
  \int_{\mathscr{C}_\text{s}} \left(
    \lambda_{\text{s},\theta^*}(s,T)
    \log\frac{\lambda_{\text{s},\theta^*}(s,T)}{\lambda_{\text{s},\theta}(s,T)} +
    \lambda_{\text{s},\theta}(s,T) - \lambda_{\text{s},\theta^*}(s,T) \right)
  \;dsP_{\theta^*}(dT).
\end{eqnarray*}
The first part of the integral can be rewritten using the expression of the split Campbell measure \eqref{eq:split:campbell:parametric}:
\begin{eqnarray*}
  \lefteqn{
    \divkl\left(\alpha_{\text{s},\theta^*,\theta^*},
    \alpha_{\text{s},\theta,\theta^*}\right)}\\
  & = &
  \int_{\mathscr{C}_\text{s}}
    \log\frac{\lambda_{\text{s},\theta^*}(s,T)}{\lambda_{\text{s},\theta}(s,T)}
  \;C^!_{\text{s},\theta^*}(ds,dT) +
  \int_{\mathscr{C}_\text{s}} \left(
    \lambda_{\text{s},\theta}(s,T) - \lambda_{\text{s},\theta^*}(s,T) \right)
  \;dsP_{\theta^*}(dT).
\end{eqnarray*}

Define
\begin{equation*}
  M_{\text{s}}(\theta,\theta^*) = \int_{\mathscr{C}_{\text{s}}} 
  \log\lambda_{\text{s},\theta}(s,T)\;C_{\text{s},\theta^*}^!(ds,dT) -
  \int_{\mathscr{C}_{\text{s}}} \lambda_{\text{s},\theta}(s,T)
  \; ds P_{\theta^*}(dT).
\end{equation*}
The extended Kullback-Leibler divergence of $\alpha_{\text{s},\theta,\theta^*}$ from $\alpha_{\text{s},\theta^*,\theta^*}$ can be written as
\begin{equation*}
  \divkl\left(\alpha_{\text{s},\theta^*,\theta^*},
    \alpha_{\text{s},\theta,\theta^*}\right) = 
  M_{\text{s}}(\theta^*,\theta^*) -
  M_{\text{s}}(\theta,\theta^*).
\end{equation*}
Only the term $M_{\text{s}}(\theta,\theta^*)$ is informative about the divergence of $\alpha_{\text{s},\theta,\theta^*}$ from $\alpha_{\text{s},\theta^*,\theta^*}$.  Note that $M_{\text{s}}$ does not have a constant sign. Instead, for any given $\theta^*$, $M_{\text{s}}(.,\theta^*)$ has a global maximum at $\theta^*$.

Now suppose that $\theta^*$ is unknown and that a realization $T$ of $P_{\theta^*}$ is available instead. In view of definition (\ref{eq:split:campbell}), the first integral of $M_{\text{s}}$ can be estimated by
\begin{equation*}
  \sum_{m\in\mathbb{M}_T} \log\lambda_{\text{s},\theta}(m^{-1},mT).
\end{equation*}
From the equivalence \eqref{eq:rel:papangelou:split:merge}, one can rewrite the estimator as
\begin{equation*}
  -\sum_{m\in\mathbb{M}_T} \log\lambda_{\text{m},\theta}(m,T).
\end{equation*}
Deriving an estimator of the second integral of $M_{\text{s}}$ is straightforward. The mean under $P_{\theta^*}$ of the integral
\begin{equation*}
  \int_{\mathbb{S}_T} \lambda_{\text{s},\theta}(s,T)
  \; ds
\end{equation*}
is equal to the second integral involved in the expression of $M_{\text{s}}(\theta,\theta^*)$ above.
  
Similarly considering the extended KL divergence of $\alpha_{\text{f},\theta,\theta^*}$ from $\alpha_{\text{f},\theta^*,\theta^*}$ leads to the function:
\begin{equation*}
  M_{\text{f}}(\theta,\theta^*) = \int_{\mathcal{T}}\int_{\mathbb{F}_T} 
  \log\lambda_{\text{f},\theta}(f,T)\;C_{\text{f},\theta^*}^!(df,dT) -
  \int_{\mathcal{T}}\int_{\mathbb{F}_T} \lambda_{\text{f},\theta}(f,T)
  \; df P_{\theta^*}(dT).
\end{equation*}
Furthermore, $M_{\text{f}}(\theta,\theta^*)$ can be estimated without bias by
\begin{equation*}
  -\sum_{f\in\mathbb{F}_T} \log\lambda_{\text{f},\theta}(f,T) -
  \sum_{f\in\mathbb{F}_T} \lambda_{\text{f},\theta}(f,T).
\end{equation*}

We can now state a first general result about the mean log-pseudolikelihood
\begin{proposition} \label{prop:ttessel:expr:mean:lpl}
If $\mathbf{T}\sim P_{\theta^*}$, then for every $\theta\in\Theta$
\begin{eqnarray}
  \nonumber
  \lefteqn{\mean_{\theta^*}\lpl(\theta|\mathbf{T})}\\
  & = &
  \label{eq:mean:lpl:ttessel:m}
  M_{\text{s}}(\theta,\theta^*)+
  M_{\text{f}}(\theta,\theta^*),\\
  & = &
  \label{eq:mean:lpl:ttessel:m:kl}
  M_{\text{s}}(\theta^*,\theta^*)+
  M_{\text{f}}(\theta^*,\theta^*)-
  \divkl\left(\alpha_{\text{s},\theta^*,\theta^*}, \alpha_{\text{s},\theta,\theta^*}\right) -
  \divkl\left(\alpha_{\text{f},\theta^*,\theta^*}, \alpha_{\text{f},\theta,\theta^*}\right).
\end{eqnarray}
\end{proposition}
\begin{proof}
The Proposition has already been proved under the assumption that $\alpha_{\text{s},\theta,\theta^*}\gg\alpha_{\text{s},\theta^*,\theta^*}$ and $\alpha_{\text{f},\theta,\theta^*}\gg\alpha_{\text{f},\theta^*,\theta^*}$.
Let us consider the case where either $\alpha_{\text{s},\theta,\theta^*}$ does not dominate $\alpha_{\text{s},\theta^*,\theta^*}$ or $\alpha_{\text{f},\theta,\theta^*}$ does not dominate $\alpha_{\text{f},\theta^*,\theta^*}$. For instance, consider the first case. There exists a non-null subset, with respect to the measure $dsP_{\theta^*}(dT)$, of $\mathscr{C}_{\text{s}}$  where $\lambda_{\text{s},\theta}$ cancels while $\lambda_{\text{s},\theta^*}$ does not. In view of expression \eqref{eq:split:campbell:hereditary}, this is equivalent with the existence of a non-null subset, with respect to the measure $C_{\text{s},\theta^*}^!$, where $\lambda_{\text{s},\theta}$ cancels. Hence if $\alpha_{\text{s},\theta,\theta^*}$ does not dominate $\alpha_{\text{s},\theta^*,\theta^*}$, then $M_{\text{s}}(\theta,\theta^*)$ is equal to $-\infty$. Since, by definition, $\divkl\left(\alpha_{\text{s},\theta^*,\theta^*},\alpha_{\text{s},\theta,\theta^*}\right)=\infty$, the equality between (\ref{eq:mean:lpl:ttessel:m}) and  (\ref{eq:mean:lpl:ttessel:m:kl}) holds. A similar rationale applies for flips instead of splits.
\end{proof}
A straigthforward consequence of Proposition \ref{prop:ttessel:expr:mean:lpl} is the following.
\begin{corollary}
  If $\mathbf{T}\sim P_{\theta^*}$, then for every $\theta\in\Theta$,  
  \begin{equation*}
    \mean_{\theta^*} \lpl(\theta|\mathbf{T}) \leq
    \mean_{\theta^*} \lpl(\theta^*|\mathbf{T})
  \end{equation*}
\end{corollary}
\begin{proof}
In view of equations (\ref{eq:mean:lpl:ttessel:m}--\ref{eq:mean:lpl:ttessel:m:kl}), we have for every $\theta\in\Theta$
  \begin{equation} \label{eq:ttessel:mean:lpl:kl}
    \mean_{\theta^*} \lpl(\theta|\mathbf{T}) =
    \mean_{\theta^*} \lpl(\theta^*|\mathbf{T}) -
    \divkl\left(\alpha_{\text{s},\theta^*,\theta^*}, \alpha_{\text{s},\theta,\theta^*}\right) -
    \divkl\left(\alpha_{\text{f},\theta^*,\theta^*}, \alpha_{\text{f},\theta,\theta^*}\right).
  \end{equation}
  Since the extended Kullback-Leibler divergences are non-negative, the inequality follows.
\end{proof}
  
The following result shows that $-(M_{\text{s}}(.,\theta^*)+M_{\text{f}}(.,\theta^*))$ is a contrast function in the sense that its global minimum at $\theta^*$ is strict.
\begin{proposition}
  If $\mathbf{T}\sim P_{\theta^*}$, then $\theta^*$ is a strict global maximum of the mean log-pseudolikelihood:
  \begin{equation*}
    \forall \theta\in\Theta\setminus\{\theta^*\},\quad
    \mean_{\theta^*} \lpl(\theta|\mathbf{T}) <
    \mean_{\theta^*} \lpl(\theta^*|\mathbf{T}).
  \end{equation*}
\end{proposition} 
\begin{proof}
  It remains to show that
  \begin{equation*}
    \mean_{\theta^*} \lpl(\theta|\mathbf{T}) =
    \mean_{\theta^*} \lpl(\theta^*|\mathbf{T}) \Rightarrow
    \theta=\theta^*.
  \end{equation*}
  Consider a $\theta$ such that the mean log-pseudolikelihoods are equal. First, observe that in view of Equation \eqref{eq:ttessel:mean:lpl:kl}, both Kullback-Leibler divergences must cancel. This implies that $\alpha_{\text{s},\theta,\theta^*}$ (respectively $\alpha_{\text{f},\theta,\theta^*}$) must dominate $\alpha_{\text{s},\theta^*,\theta^*}$ (respectively $\alpha_{\text{f},\theta^*,\theta^*}$), otherwise one of the KL-divergences is infinite. Hence the mean log-pseudolikelihoods can be equal only if $\alpha_{\text{s},\theta,\theta^*}\gg \alpha_{\text{s},\theta^*,\theta^*}$ and $\alpha_{\text{f},\theta,\theta^*}\gg \alpha_{\text{f},\theta^*,\theta^*}$. From now on, we focus on that case.

  The Kullback-Leibler divergence of $\alpha_{\text{s},\theta,\theta^*}$ from $\alpha_{\text{s},\theta^*,\theta^*}$ cancels if and only if $\alpha_{\text{s},\theta,\theta^*}=\alpha_{\text{s},\theta^*,\theta^*}$ except on a $\alpha_{\text{s},\theta,\theta^*}$-null subset $E\subset\mathscr{C}_{\text{s}}$. In view of definition \eqref{eq:def:alpha:split}, $E$ is $\alpha_{\text{s},\theta,\theta^*}$-null if and only there exists a pair $(E_1,E_2)$ of subsets of $\mathscr{C}_{\text{s}}$ such that $E=E_1\cup E_2$, $E_1$ is null with respect to the measure $ds P_{\theta^*}(dT)$ and $\lambda_{\text{s},\theta}$ cancels on $E_2$. The domination $\alpha_{\text{s},\theta,\theta^*}\gg \alpha_{\text{s},\theta^*,\theta^*}$ implies that $\lambda_{\text{s},\theta^*}$ must cancel too on $E_2$. Therefore $\alpha_{\text{s},\theta,\theta^*}$ and $\alpha_{\text{s},\theta^*,\theta^*}$ coincide on $E_2$. More generally, both measures equal except on $E_1$ which is null with respect to $dsP_{\theta^*}(dT)$. In turn, this implies that $\lambda_{s,\theta}$ and $\lambda_{\text{s},\theta^*}$ equal for $P_{\theta^*}$-almost all $T\in\mathcal{T}$ and almost all splits $s\in\mathbb{S}_T$. Hence, $\lambda_{\text{s},\theta^*}$ is a split Papangelou conditional intensity for $P_{\theta^*}$.

  Similarly, it can be shown that $\lambda_{\text{f},\theta}$ is a flip Papangelou conditional intensity for $P_{\theta^*}$. From Proposition \ref{prop:sf:papangelou:determines:dist}, it follows that $h_\theta$ is a density of $P_{\theta^*}$ and that it equals $h_{\theta^*}$ $\mu$-almost everywhere. Since the model is supposed to be identifiable, $\theta$ must equal $\theta^*$.
\end{proof}
\subsection{The canonical exponential family case}
\label{sec-2-3}

\label{sec:ttessel:exponential}
Consider the case of a canonical exponential family:
\begin{equation}
  \label{eq:tessel:exponential:family}
  P_\theta(dT) \propto \exp\left(\theta^T t(T)\right)\, \mu(dT),
\end{equation}
wher $t(T)\in\R^d$ is a vector of statistics of tessellation $T$. Then we have
\begin{eqnarray*}
  \lambda_{\text{s},\theta}(s,T) & = &
  \exp\left(\theta^T t(s,T)\right),\quad t(s,T) = t(sT)-t(T),\\
  \lambda_{\text{m},\theta}(m,T) & = &
  \exp\left(\theta^T t(m,T)\right),\quad t(m,T) = t(mT)-t(T),\\
  \lambda_{\text{f},\theta}(f,T) & = &
  \exp\left(\theta^Tt(f,T)\right),\quad t(f,T) = t(fT)-t(T).
\end{eqnarray*}
The log-pseudolikelihood writes
\begin{multline}
  \label{eq:exponential:tessel:pseudo}
  \lpl(\theta|T) = -
  \theta^T \sum_{m\in\mathbb{M}_T} t\left(m,T\right) -
  \int_{\mathbb{S}_T} \exp\left(\theta^T t\left(s,T\right)\right)\; ds -\\
  \theta^T \sum_{f\in\mathbb{F}_T} t\left(f,T\right) -
  \sum_{f\in\mathbb{F}_T} \exp\left(\theta^T t\left(f,T\right)\right).
\end{multline}
The gradient of the log-pseudolikelihood has the following form
\begin{multline}
  \label{eq:exponential:tessel:pseudo:gradient}
  \nabla\lpl(\theta|T) = -
  \sum_{m\in\mathbb{M}_T} t\left(m,T\right) -
  \int_{\mathbb{S}_T} t(s,T)\exp\left(\theta^T t\left(s,T\right)\right)\; ds -\\
  \sum_{f\in\mathbb{F}_T} t\left(f,T\right) -
  \sum_{f\in\mathbb{F}_T} t(f,T)\exp\left(\theta^T t\left(f,T\right)\right).
\end{multline}
The Hessian matrix of $\lpl$ has elements of the form
\begin{multline*}
  H_{ij}(\lpl)(\theta|T) = 
  -\int_{\mathbb{S}_T} t_i(s,T)t_j(s,T)\exp(\theta^Tt(s,T))\;ds
  -\\\sum_{f\in\mathbb{F}_T} t_i(f,T)t_j(f,T)\exp(\theta^Tt(f,T)).
\end{multline*}
That is
\begin{multline*}
  H(\lpl)(\theta|T) = -
  \int_{\mathbb{S}_T}t(s,T)t(s,T)^T \exp(\theta^Tt(s,T))\;ds-\\
  \sum_{f\in\mathbb{F}_T}t(f,T)t(f,T)^T \exp(\theta^Tt(f,T))
\end{multline*}
As a consequence the log-pseudolikelihood is concave. Strict concavity holds if for every $y\in\R^d$
\begin{multline*}
  y^TH(\lpl)(\theta|T)y = - 
  \int_{\mathbb{S}_T} \left(y^Tt(s,T)\right)^2 \exp(\theta^Tt(s,T))\;ds - \\
  \sum_{f\in\mathbb{F}_T} \left(y^Tt(f,T)\right)^2 \exp(\theta^Tt(f,T)) < 0.
\end{multline*}
Define
\begin{equation}
  \label{eq:def:sstar:fstar}
  \mathbb{S}_T^*=\left\{s\in\mathbb{S}_T:t(s,T)>-\infty\right\},\quad
  \mathbb{F}_T^*=\left\{f\in\mathbb{F}_T:t(f,T)>-\infty\right\}.
\end{equation}
In the quadratic form above, the integral can be restricted to $\mathbb{S}_T^*$ and the sum to $\mathbb{F}_T^*$. The log-pseudolikelihood is strictly concave if at least one of the two following conditions is fulfilled.
\begin{condition}\label{cond:concavity:flip}
  The subset $t(\mathbb{F}_T^*,T)$ of $\R^d$ is not included in a vector hyperplane.
\end{condition}
\begin{condition}\label{cond:concavity:split}
  For any subset $B$ of $\mathbb{S}_T^*$ almost equal to $\mathbb{S}_T^*$, $t(B,T)$ is not included in a vector hyperplane of $\R^d$.
\end{condition}
Above $B$ is almost equal to $\mathbb{S}_T^*$ if and only if
\begin{equation*}
  \int_{\mathbb{S_T}} I_{\mathbb{S}_T^*\setminus B}(s)\;ds = 0.
\end{equation*}
Condition \ref{cond:concavity:flip} ensures that for every $y\in\R^d$, there exists at least a flip $f\in\mathbb{F}_T^*$ such $t(f,T)$ is not orthogonal to $y$. Similarly Condition \ref{cond:concavity:split} ensures that there exists a non-null subset of $\mathbb{S}_T^*$ where $t(s,T)$ is not orthogonal to $y$.
\begin{example}
  \label{example:crtt}
  Consider the distribution defined by the density
  \begin{equation}
    \label{def:model:crtt}
    h_\theta(T) \propto \exp\left(\nseint{T}\theta\right)
  \end{equation}
  where $\nseint{T}$ refers to the number of internal segments of $T$. Above the density $h_{\theta}$ is defined up to a unknown normalizing constant depending on $\theta$. This distribution defines a CRTT with intensity controlled by parameter $\theta$. In \cite{ref/1242}, the CRTT model is introduced together with another parametrization based on $\tau=\exp(\theta)$. The parameter $\tau$ can be considered as a scaling parameter: increasing $\tau$ is equivalent to increasing the tessellated domain $D$. Let us come back to parametrization \eqref{def:model:crtt} which defines a canonical exponential family with
  \begin{equation*}
    t(T) = \nseint{T},
  \end{equation*}
  and thus
  \begin{eqnarray*}
    t(s,T) & = & 1,\\
    t(m,T) & = & -1,\\
    t(f,T) & = & 0.
  \end{eqnarray*}
  The log-pseudolikelihood writes
  \begin{equation*}
     \lpl(\theta|T) = \left|\mathbb{M}_T\right|\theta-
     \frac{u(T)}{\pi}\exp(\theta) - \left|\mathbb{F}_T\right|.
  \end{equation*}
  The log-pseudolikelihood can be rewritten in terms of the number of internal segments of $T$:
  \begin{equation*}
     \lpl(\theta|T) = \nnbseint{T}\theta -
     \frac{u(T)}{\pi}\exp(\theta) - 2\nbseint{T},
  \end{equation*}
  where $\nnbseint{T}$ is the number of non-blocking internal segments of $T$ and $\nbseint{T}$ the number of blocking internal segments.
  The function above is strictly concave (negative second derivative). It has therefore a unique maximum. In order to identify the maximum, one solves
  \begin{equation*}
    \lpl'(\hat{\theta}|T) = \nnbseint{T} - \frac{u(T)}{\pi}\exp(\hat{\theta}) = 0. 
  \end{equation*}
  This yields
  \begin{equation*}
    \hat{\theta} = \log \frac{\nnbseint{T}\pi}{u(T)}.
  \end{equation*}
  That is
  \begin{equation*}
    \hat{\tau} = \frac{\nnbseint{T}\pi}{u(T)}.
  \end{equation*}

  Note that for the CRTT model, claimed to be the analog of the Poisson point process, the pseudolikelihood maximum estimator is not the same as the likelihood estimator. Indeed for the CRTT model the likelihood has the form
  \begin{equation*}
    \frac{\tau^{\nseint{T}}}{c(\tau)},
  \end{equation*}
  where $c(\tau)$ is a normalization constant (unknown so far). It follows that the likelihood maximum estimator of $\tau$ is a function (still unknown) of the number $\nseint{\tau}$ of internal segments of the tessellation $T$.

  The estimators derived from the likelihood and the pseudolikelihood do not coincide for completely random T-tessellations contrary to what happens for Poisson point processes. 
\end{example}
\begin{example}
  \label{example:areas}
  Consider the distribution defined (up to an unknown nomalizing constant) by the density
  \begin{equation}
    \label{eq:density:areas}
    h_\theta(T) \propto \exp\left(\nseint{T}\theta_1-a^2(T)\theta_2\right),
  \end{equation}
  where $a^2(T)$ is the sum of $T$ cell squared areas. Given the number of cells the statistics $a^2(T)$ is large when the cells have heterogeneous areas. Large values of $\theta_2$ favour tessellations with homogeneous cells (with respect to their areas). Again we have a canonical exponential family with
  \begin{equation*}
    t(T) = \begin{pmatrix} \nseint{T}\\ -a^2(T)\end{pmatrix}.
  \end{equation*}
  Expressions for $t(s,T)$, $t(m,T)$ and $t(f,T)$ can be worked out. But this does not lead to a simple analytical expression of the pseudolikelihood.
\end{example}
\subsection{Algorithm for approximating the pseudolikelihood maximum}
\label{sec-2-4}

\label{sec:nois:algorithm}
A simple approach for approximating the log-pseudolikelihood consists in discretizing the integral on $\mathbb{S}_T$ involved in the pseudolikelihood. Below we consider the framework of a canonical exponential family.
The log-pseudolikelihood and its gradient have the forms given by Equations \eqref{eq:exponential:tessel:pseudo} and \eqref{eq:exponential:tessel:pseudo:gradient}.

Let $S$ be a finite set of so-called \emph{dummy} splits of $T$. The number of elements of $S$ is denoted $|S|$. The log-pseudolikelihood \eqref{eq:exponential:tessel:pseudo} is replaced by
\begin{multline}
  \label{eq:exponential:tessel:discrete:pseudo}
  \lpl_{\text{d}}(\theta|T;S) = -
  \theta^T \sum_{m\in\mathbb{M}_T} t\left(m,T\right) -
  \frac{u(T)}{\pi|S|}
  \sum_{s\in S} \exp\left(\theta^T t\left(s,T\right)\right) -\\
  \theta^T \sum_{f\in\mathbb{F}_T} t\left(f,T\right) -
  \sum_{f\in\mathbb{F}_T} \exp\left(\theta^T t\left(f,T\right)\right).
\end{multline}
The approximation concerns only the integral over $\mathbb{S}_T$ involved in \eqref{eq:exponential:tessel:pseudo} which is replaced by a weighted discrete sum. All other terms involving merges and flips are exact.
Notice that if $S$ is a sample of splits drawn uniformly from $\mathbb{S}_T$, then $\lpl_{\text{d}}(\theta|T;S)$ is an unbiased estimator of $\lpl(\theta|T)$. 

\begin{remark}
In the special case of the CRTT model, the approximation is exact. Since $t(s,T)=1$ for any split $s$, the sum over $S$ in \eqref{eq:exponential:tessel:discrete:pseudo} writes
\begin{equation*}
  \sum_{s\in S} \exp\left(\theta^T t\left(s,T\right)\right) =
  \left|S\right| \exp\left(\theta\right).
\end{equation*}
Thus the term involving dummy splits becomes
\begin{equation*}
  \frac{u(T)}{\pi}\exp(\theta).
\end{equation*}
Compare with Example \ref{example:crtt}.
\end{remark}

A straightforward approximation of the log-pseudolikelihood gradient is  
\begin{multline}
  \label{eq:exponential:tessel:discrete:pseudo:gradient}
  \nabla\lpl_{\text{d}}(\theta|T;S) = -
  \sum_{m\in\mathbb{M}_T} t\left(m,T\right) -
  \frac{u(T)}{\pi|S|}
  \sum_{s\in S} t(s,T)\exp\left(\theta^T t\left(s,T\right)\right) -\\
  \sum_{f\in\mathbb{F}_T} t\left(f,T\right) -
  \sum_{f\in\mathbb{F}_T} t(f,T)\exp\left(\theta^T t\left(f,T\right)\right).
\end{multline}
Similarly, the log-pseudolikelihood Hessian can be approximated by
\begin{multline}
  \label{eq:exponential:tessel:discrete:pseudo:hessian}
  H(\lpl_\text{d})(\theta|T;S) = - \frac{u(T)}{\pi|S|}
  \sum_{S}t(s,T)t(s,T)^T \exp(\theta^Tt(s,T))-\\
  \sum_{f\in\mathbb{F}_T}t(f,T)t(f,T)^T \exp(\theta^Tt(f,T))
\end{multline} 

\begin{algorithm}
\caption{Algorithm NOIS (Newton optimization and increasing splitting) for approximating the pseudolikelihood maximum.}
\label{algo:nois}
\begin{algorithmic}[1]
  \Require a T-tessellation $T$, the stepsize $\epsilon$ to be used in $\theta$'s update, the tolerance parameter $\delta$ 
  \State $m\gets$ number of non-blocking segments of $T$
  \State $\theta\gets$ null vector
  \State Compute $\sum t(m,T)$ for $m\in\mathbb{M}_T$
  \State Compute $t(f,T)$ for all $f\in\mathbb{F}_T$
  \State Compute the sum of the $t(f,T)$'s
  \State $S\gets$ sample of  $m$ independent and uniform splits $s$ of $T$
  \State Compute the $t(s,T)$'s
  \Repeat
    \State $G\gets\nabla\lpl_{\text{d}}(\theta|T;S)$ as given by \eqref{eq:exponential:tessel:discrete:pseudo:gradient}
\State $H\gets H(\lpl_\text{d})(\theta|T;S)$ as given by 
\eqref{eq:exponential:tessel:discrete:pseudo:hessian}

    \State $\theta\gets\theta+\epsilon H^{-1}G$
    \State Add $m$ independent and uniform splits of $T$ to $S$
    \State Update the list of $t(s,T)$'s where $s\in S$
    \State $L\gets \lpl_{\text{d}}(\theta|T;S)$ as given by \eqref{eq:exponential:tessel:discrete:pseudo}
    \State $\Delta\gets$ variation of $L$
  \Until{$\left|\Delta\right| \leq
         \delta \left(\left|L\right| + \delta \right).$}
\end{algorithmic}
\end{algorithm}

The integral on splits in the log-pseudolikelihood \eqref{eq:exponential:tessel:pseudo} can be written as an expectation with respect to the uniform distribution on the space $\mathbb{S}_T$ of splits. The estimation criterion consists of deterministic terms and of an expectation with respect to a random variable that can be simulated. Stochastic approximation provides a number of algorithms for optimizing such type of criterion. However the present framework is rather simple here. The expectation is taken under a fixed distribution not depending on the searched parameter.

From stochastic approximation we keep the alternation between the update of the parameter and the estimation of the expectation. An iteration of the algorithm consists of the two following stages:
\begin{itemize}
\item Update of $\theta$ according to Newton's optimization method based on the log-pseudolikelihood Hessian approximation.
\item Increment of the sample $S$ of dummy splits.
\end{itemize}

The full procedure is detailed in Algorithm \ref{algo:nois}. It is referred to as NOIS (Newton optimization and increasing splitting) algorithm below.
Some remarks:
\begin{itemize}
\item The memory space used by the algorithm for storing the $t(s,T)$'s grows along iterations.
\item The time for computing
  \begin{equation*}
    \sum_{s\in S} \exp\left(\theta^T t\left(s,T\right)\right)
  \end{equation*}
  grows along iterations too.
\item One could add one dummy split to $S$ at each iteration, but the convergence of the sum on dummy splits is slow and the trajectories of the $\theta_n$'s are then rather irregular.
\item The stopping criterion is inspired by the R function optim.
\item When the parametric family includes the CRTT distributions, one may start with the maximum pseudolikelihood estimate under the CRTT model. Without loss of generality, one may assume that the first component of $t(T)$ is equal to $\nseint{T}$. The initial value of $\theta$ writes:
  \begin{equation*}
    \left(
          \log\left(\nnbseint{T}\pi\right)/u(T),0,\ldots,0
        \right)^T.
  \end{equation*}
\end{itemize}

\begin{example}
\label{example:angles}
Consider the model with the following non-normalized density 
\begin{equation}
  \label{eq:density:model:angles}
  h_\theta(T) \propto \exp
  \left(
    \theta_1\nseint{T} - 
    \theta_2\sum_{v\in V(T)}\left(\frac{\pi}{2}-\phi(v)\right)
  \right),
\end{equation}
where $0\leq\phi(v)\leq\pi/2$ denotes the acute angle between the two segments meeting at vertex $v$. When $\theta_2>0$, T-tessellations with almost rectangular cells are favoured by the model. A realization of such a model with $\theta_1=\numprint{2.49}$ and $\theta_2=\numprint{2.5}$ is shown in Figure \ref{fig:pseudo:discrete:angles} (top). Figure \ref{fig:pseudo:discrete:angles} (bottom) shows an example of parameter trajectory obtained with the NOIS algorithm. 
Some details:
\begin{itemize}
\item The observed tessellation was generated by a SMF Markov chain starting from the empty tessellation and after a burnin stage of \numprint{1000} iterations.
\item A sample of \numprint{1000} dummy splits was used for computing the discrete approximation of the pseudolikelihood and of its gradient.
\item As a starting point for the optimization, the pseudolikelihood estimator assuming the CRTT model was used.
\item Ten iterations have been performed before the stopping criterion was met with $\delta=\numprint{0.01}$. Since the T-tessellation used for inference had \numprint{31} non-blocking segments, the NOIS algorithm only used $10\times 31=310$ dummy splits.
\end{itemize}
\end{example}

\begin{figure}[ptbh]
\centering
\includegraphics[width=4.7in]{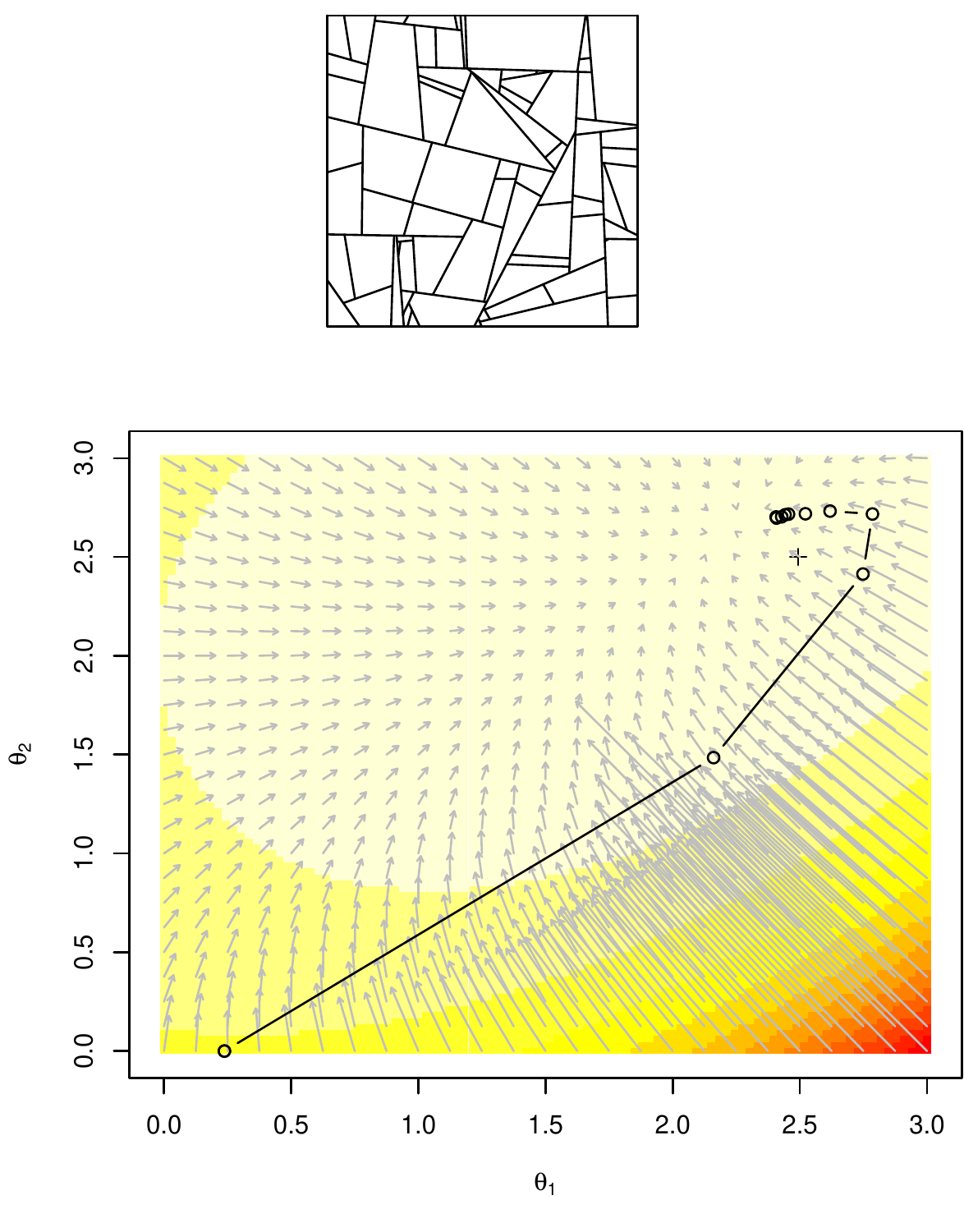}
\caption{\label{fig:pseudo:discrete:angles}Top: T-tessellation generated by the model with penalization on acute angles between segments. Bottom: maximum pseudolikelihood estimation. Background: discrete approximation of the log-pseudolikelihood (increasing from red to white). Grey arrows: approximated log-pseudolikelihood gradient. Circles: intermediate estimates found by NOIS algorithm. Cross: true parameters.}
\end{figure}
\subsection{Simulation study}
\label{sec-2-5}

\label{sec:simulations}
At the time being, it is not possible to investigate the asympotic behaviour of the maximum pseudolikelihood estimator. Gibbs models of T-tessellations are defined only in bounded domains. Extensions to the whole plane, existence and unicity of Gibbs measures are still open questions. In order to provide some hints on the behaviour of the maximum pseudolikelihood estimator, we proceeded to a simulation study. Three models of random T-tessellations and three domain sizes for each model are considered.

The first model is the CRTT model defined in Example \ref{example:crtt}. The second model is defined in Example \ref{example:angles} and favour T-tessellations with perpendicular segments. The last model favours T-tessellation with homogeneous cell areas and is defined in Example \ref{example:areas}. The three models are referred to as CRTT, angle and area models below. Parameter values are given in Table \ref{tab:sim:data}. Each model has been simulated on three squared domains with varying sizes. The domain sizes were chosen in order to emphasize their effect on the estimator precision. Domain side lengths are provided in Table \ref{tab:sim:data}.

\begin{table}[htb]
\caption{Simulation settings. For each model, $\theta$ values are given. Tessellated domains are squares with side lengths $w_1$, $w_2$ and $w_3$. Numbers of iterations for the burnin stage and between sampled tessellations (sampling period) are also provided.} \label{tab:sim:data}
\begin{center}
\begin{tabular}{r*{8}{d{1}}}
 model  &  \colhead{$\theta_1$}  &  \colhead{$\theta_2$}  &  \colhead{$w_1$}  &  \colhead{$w_2$}  &  \colhead{$w_3$}  &  \colhead{burnin}  &  \colhead{period}  \\
\hline
 CRTT   &                0.64  &                      &            2.5  &           1.75  &              1  &           12500  &            3704  \\
 angle  &                2.49  &                 2.5  &            2.5  &           1.75  &            1.4  &           30000  &            9473  \\
 area   &                0.53  &               835.2  &            2.8  &            1.9  &            1.5  &           11000  &            7223  \\
\end{tabular}
\end{center}
\end{table}

The maximum pseudolikelihood estimator has been computed on samples of 500 T-tessellations for each model and domain size. In order to generated the T-tessellation, SMF Markov chains has been  used as described in \cite{ref/1242}. As for any Metropolis-Hastings algorithm, burnin and sampling period (number of Markov chain iterations between consecutive sampled items) have to be determined.

The burnin duration was tuned empirically by monitoring the evolution of the model energy. It was chosen as the number of iterations required for energy stabilization (visual assessment). 

Concerning the sampling period, an automatic procedure was used. Along an SMF chain curse, segments are born (splits), change their length (flips) or die (merge). Thus a segment has a survival time expressed in terms of number of iterations. As a criterion, the renewal rate of segments was considered. The sampling period was chosen long enough so that \numprint[\%]{75} of segments living at the beginning of the period had died at the end of the period. Sampling periods were approximated using the RLiTe function \funname{getSMFSamplingPeriod}.

Note that burnin duration and sampling periods were tuned for each model on the largest domain to be considered. Simulations on smaller domains were performed using the SMF Markov chains with parameters chosen for the largest domains.

First sample times determined from the burnin duration and the sampling period are shown for the three models simulated in their largest domains in Figure \ref{fig:all:sampling}.

\begin{figure}[ptbh]
\centering
\includegraphics[width=4.8in]{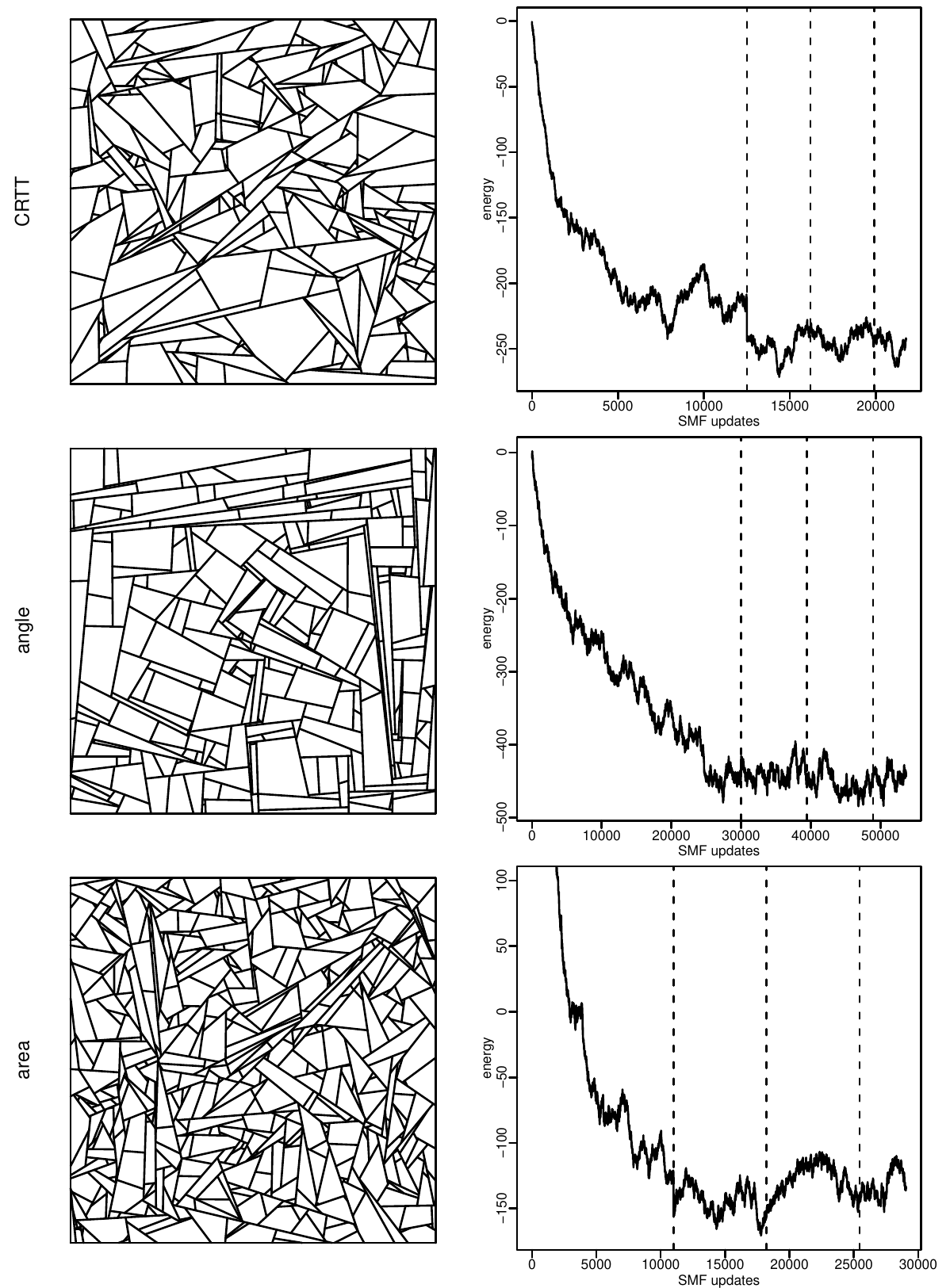}
\caption{\label{fig:all:sampling}Simulation of the CRTT (top row), angle (middle row) and area (bottom row) models. Left: first tessellation sampled at the end of the burnin stage. Right: energy evolution during the first SMF iterations. Vertical lines: iterations where tessellations are sampled.}
\end{figure}

\begin{table}[htb]
\caption{Estimation settings used by the NOIS algorithm. For each model, the values of the tolerance parameter $\delta$ and the maximal number of NOIS updates.} \label{tab:estim:data}
\begin{center}
\begin{tabular}{r*{2}{d{3}}}
 model  &  \colhead{tolerance}  &  \colhead{max.\ iterations}  \\
\hline
 CRTT   &                     &                         0  \\
 angle  &              0.005  &                       150  \\
 area   &             -0.005  &                       100  \\
\end{tabular}
\end{center}
\end{table}

In order to compute estimates using the LiTe class \classname{PLInferenceNOIS}, one needs to specify the tolerance and the maximal number of iterations. Values are given in Table \ref{tab:estim:data}. Estimation for the CRTT model is apart since the maximum pseudolikelihood has an explicit simple form. It can be computed using method \funname{Run} of class \classname{PLInferenceNOIS} with a null maximum number of iterations. The maximal number of iterations for the angle model was chosen large enough so that the stopping criterion could  be met in most cases. Concerning the area model, it turned out that the stopping criterion was met within very few iterations (sometimes only one for the smallest model). This resulted in rather biased estimates for the parameter $\theta_1$. This poor performance may be due to flat pseudolikelihood functions as illustrated below. In order to bypass this problem, the tolerance parameter $\delta$ is chosen negative so that the stopping criterion cannot be met and the maximum number of NOIS iterations is chosen large.

\begin{figure}[ptbh]
\centering
\includegraphics[width=5in]{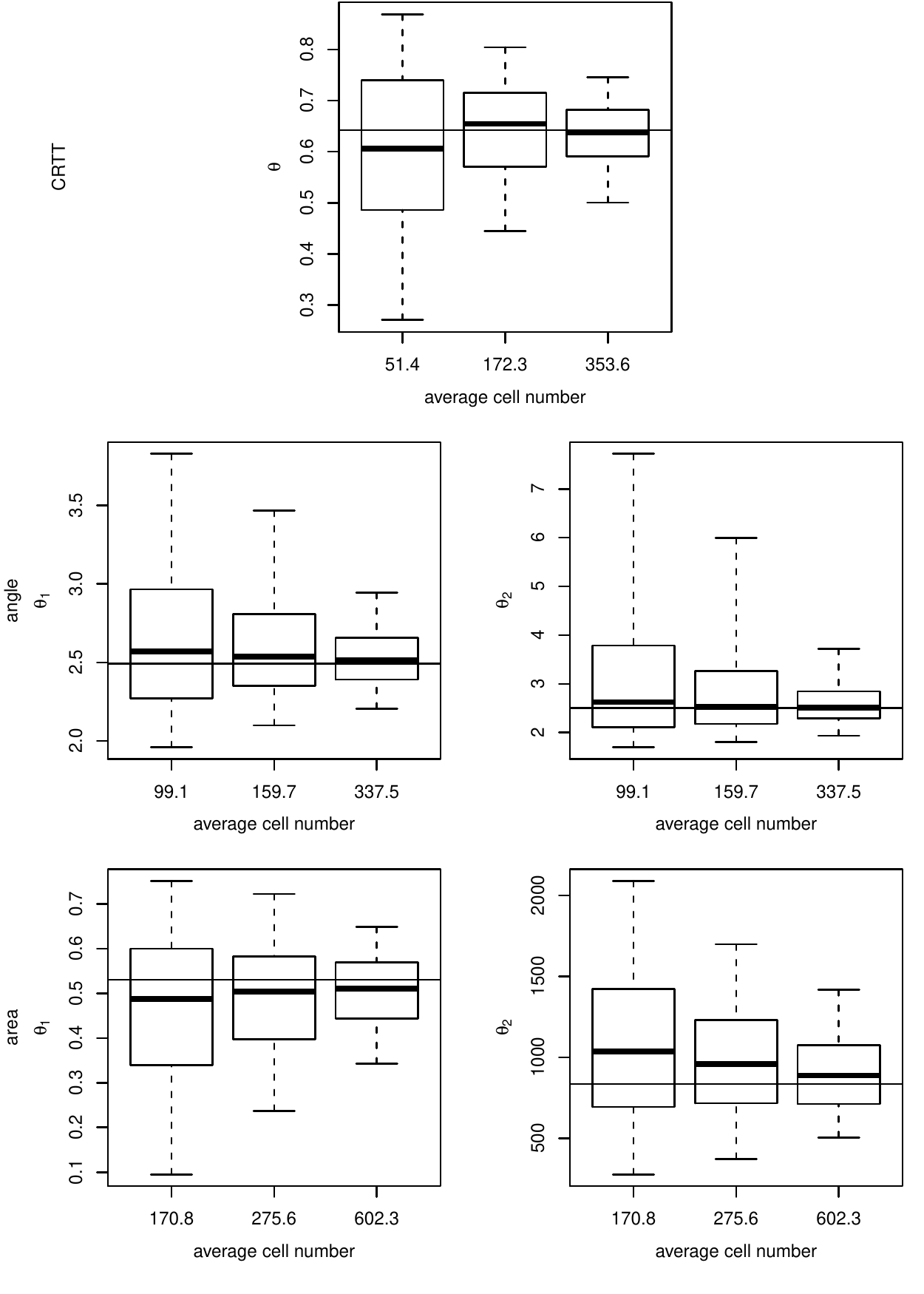}
\caption{\label{fig:boxplot:estimates:all}Maximum pseudolikelihood estimates (500 replicates) for the three models (CRTT, angle and area). Abscissa: mean number of cells for each of the three domain size. For each model and domain size, the estimate distribution is represented as a boxplot. First and last deciles are shown as the whisker ends. True parameter values represented as horizontal lines.}
\end{figure}

Estimate distributions for the CRTT model are shown in Figure \ref{fig:boxplot:estimates:all} (top row) as boxplots. Dispersion is moderate and decreases when the domain size increases. Small bias: it is not visible as soon as the total number of cells is larger than a hundred. For fifty cells (smallest domain size), a small negative bias cannot be excluded.

Estimate distributions for the angle model are shown in Figure \ref{fig:boxplot:estimates:all} (middle row). As for the CRTT model, the estimate dispersion and bias seem to decrease when the domain size increases. The estimate distributions show some asymmetry.

Let us consider the area model. As mentionned above, the stopping criterion of the NOIS algorithm does not seem to work for the area model. This problem is illustrated based on a simulated tessellation. The pseudolikelihood is approximated using \numprint{5000} dummy splits, see Figure \ref{fig:areas:flat:pseudo}. From now on this approximation is considered as exact. In the region around the maximum, the log-pseudolikelihood is rather flat: relative variations of the log-pseudolikelihood are small. As a consequence, the NOIS algorithm stops too early and find an estimate a little bit too far from the true maximum. A possible enhancement of the stopping criterion: stop when the pseudolikelihood and the estimates do not vary enough any more. An alternative solution could be to reparametrize the density \eqref{eq:density:areas} in order to get two statistics of $T$ with more comparable values.

\begin{figure}[htbp]
\centering
\includegraphics[width=4in]{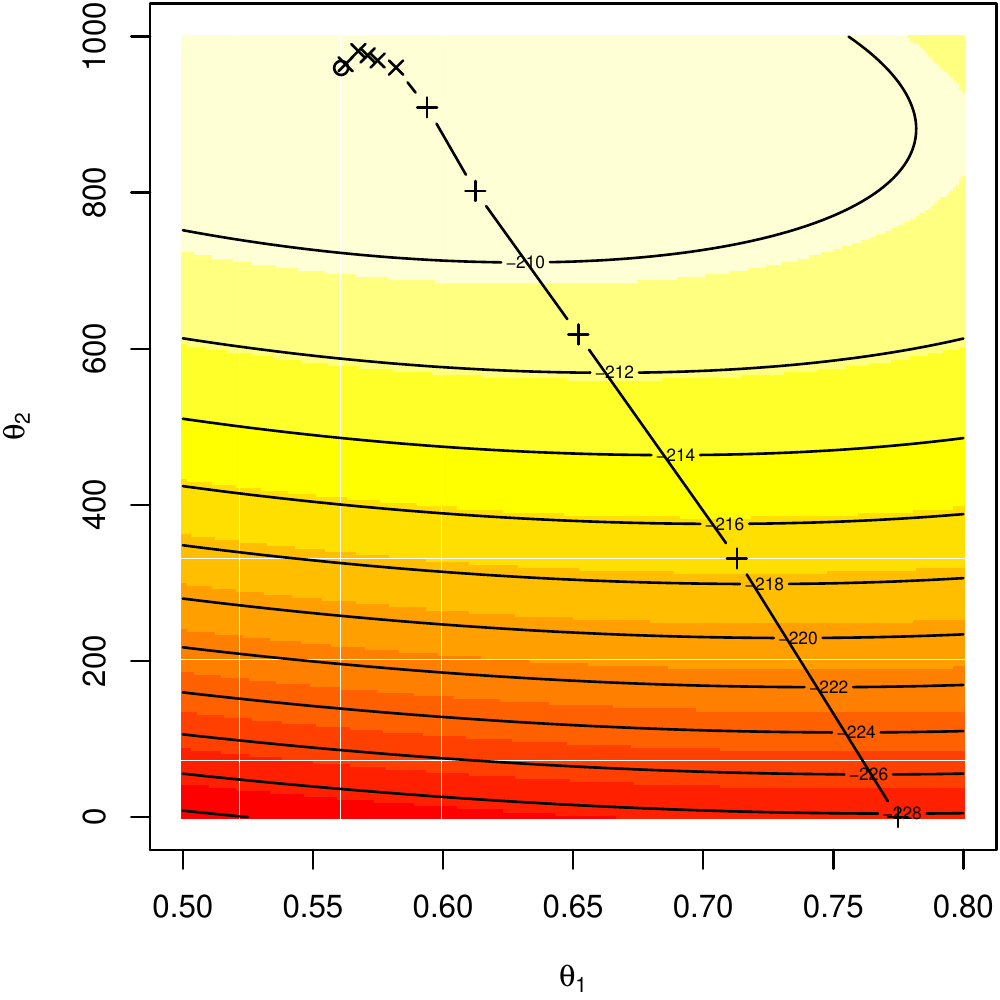}
\caption{\label{fig:areas:flat:pseudo}Pseudolikelihood of a simulated tessellation under the area model. Intermediate estimates are shown as plus signs when found before the stopping criterion was fulfilled. Further estimates are shown as crosses. The location of the true maximum is shown as a circle.}
\end{figure}

Estimation for the area model behavess like the two previous models, see Figure \ref{fig:boxplot:estimates:all} (bottom row). Both dispersion and bias decrease when the domain size increases.
\subsection{Fitting a wrong model}
\label{sec-2-6}

\label{sec:wrong:model}
Let us consider an observed T-tessellation approximating agricultural plots located around Selommes village (Loir-et-Cher department, France). This tessellation is shown in Figure \ref{fig:wrong:model} (center). The tessellated domain has a width of about \numprint[km]{5.3}. There are 211 agricultural plots (cells). Choosing a good model for representing such a tessellation is an open research problem far beyond the scope of this paper. As a toy model, consider the CRTT model. It is not plausible but rather simple. Our observed T-tessellation has 86 internal non-blocking segments. The sum of all cell perimeters is equal to \numprint[km]{213,2}. Thus the estimate of $\theta$ is equal to $\hat{\theta}=\numprint{0.24}$.

\begin{figure}[p]
\centering
\includegraphics[width=5in]{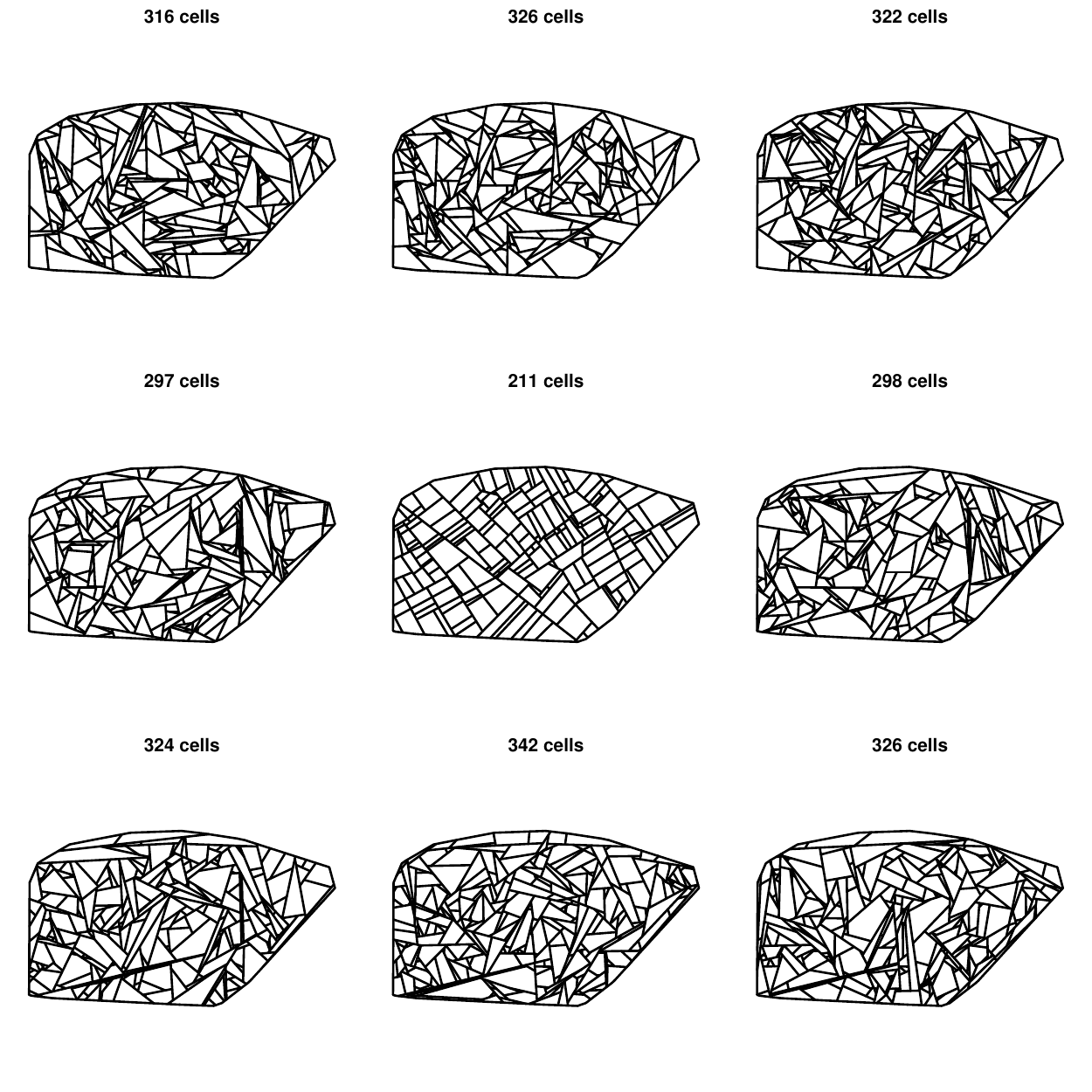}
\caption{\label{fig:wrong:model}Center: a T-tessellation supposed to be a realization of the CRTT model with unknown parameter. Around: realizations of the CRTT model fitted using the maximum pseudolikelihood estimator.}
\end{figure}

Simulations of T-tessellations under the CRTT model with parameter $\theta=\numprint{0.24}$ are shown in Figure \ref{fig:wrong:model}. The most striking feature of simulated tessellations is their heterogeneity compared to the observed T-tessellation. This is expected since the CRTT model yields completely random T-tessellation while the observed tessellation shows some kind of regularity. But there is another important difference between the observed and the simulated tessellations: there are more cells in the latter. 

This difference is unexpected. Such a behaviour does not occur with point processes. In the point process framework, the situation is as follows. Consider a given point pattern of say $n$ points. Suppose that the homogeneous Poisson model is chosen for fitting. The intensity parameter is estimated by $n$ divided by the area of the considered domain. This estimate is obtained both as the maximum pseudolikelihood and as the maximum likelihood estimators. Therefore simulating the fitted Poisson model would yield point patterns of size $n$ in expectation. Somehow pseudolikelihood fitting leads to the best Poisson point process representing the observed point pattern. In the T-tessellation framework, the fitted model seems to be biased for reasons that remain to be investigated.
\section{Pseudolikelihood for point processes}
\label{sec-3}

\label{sec:pl:point:processes}
Pseudolikelihood was first introduced by Besag for spatial random fields with discrete support \cite{ref/1243}, later for a Strauss point process \cite{ref/419}. The extension to a rather general class of models for point processes was provided by Jensen and Møller \cite{ref/1100}. General convergence results (consistancy, central limit theorem) were obtained by  Billiot et al. \cite{ref/1245}. In this report, we focus on point processes within a bounded domain denoted by $D$. The way the pseudolikelihood is introduced here is original. Previously the point process was discretized by countings on a mesh (a square lattice for example). This yields a random field with discrete support. By letting the tile size decrease to zero, one obtains the point process pseudolikelihood as a limit of the random field pseudolikelihood (the limit does not depend on the considered mesh). Similarly to random T-tessellations, one can derive the pseudolikelihood starting from the Kullback-Leibler divergence and candidate Campbell measures. Such an approach is developed below. It should be noticed that inference based on pseudolikelihood can also be considered as a particular instance of large class of methods so-called Takacs-Fiksel methods \cite{ref/1246}.

Below only the main steps of the pseudolikelihood derivations are sketched. Full details can be recovered by analogy with the pseudolikelihood derivation for random T-tessellations. 
\subsection{Finite point processes}
\label{sec-3-1}

Let $\mathscr{D}$ denote the space of finite point patterns in a planar compact domain $D$. Considering only bounded domains simplifies the presentation and makes the analogy with random T-tessellations more straightforward.

Any point process $\mathbf{X}$ defines a reduced Campbell measure $C^!$ on $D\times\mathscr{D}$ which maps any measurable function $\phi:D\times\mathscr{D}\rightarrow\R$ to
\begin{equation}
\label{eq:def:campbell}
C^!(\phi) = \mean \sum_{x\in\mathbf{X}} \phi(x,\mathbf{X}\setminus\{x\}). 
\end{equation}
The Campbell measure above is very similar to the split Campbell measure defined for random T-tessellations. A split of a T-tessellation is the insertion of a non-blocking segment. Its analog for point patterns is the insertion of a new point. A merge of a T-tessellation is the removal of a non-blocking segment. Its analog for point patterns is the removal of an (arbitrary) point of the pattern.

Let $\mu$ denote the distribution of a Poisson point process with intensity 1 on $D$. Consider a distribution  $P\ll\mu$ on $\mathscr{D}$ with density $h$ :
\begin{equation*}
  P(dX) = h(X)\; \mu(dX).
\end{equation*}
Let us suppose that $h$ is hereditary :
\begin{equation}
  \label{eq:hereditary}
  \forall x\in D, \forall X\in\mathscr{D},\quad
  h(X)=0 \Rightarrow h(X\cup\{x\})=0.
\end{equation}
Then the reduced Campbell measure is absolutely continuous with respect to the product of the Lebesgue measure on $\R^2$ with $P$ and its density denoted $\lambda$ with respect to the product measure is defined as the Papangelou conditional intensity given by
\begin{equation*}
  C^!(dx,dX) = \lambda(x;X)\;dx P(dX) 
\end{equation*}
The Papangelou conditional intensity has the following expression
\begin{equation}
  \label{eq:papangelou:hereditary}
  \lambda(x;X) = \frac{h(X\cup\{x\})}{h(X)}.
\end{equation}
Note that one can find in the literature expressions of the Papangelou conditional intensity slightly different from \eqref{eq:papangelou:hereditary}. For instance, for $x\in X$, Baddeley and Turner \cite{ref/758} define $\lambda(x;X)$ as $h(X)/h(X\setminus x)$ instead of 1. The different versions of the Papangelou conditional intensity are consistent with \eqref{eq:papangelou:hereditary} as long as they coincide with \eqref{eq:papangelou:hereditary} almost-everywhere with respect to the product of the Lebesgue measure on $\R^2$ with $P$. 
\subsection{Contrast function and pseudolikelihood}
\label{sec-3-2}

\label{sec:contrast:pl:point:processes}
Consider a parametric family of distributions $P_\theta$ on $\mathscr{D}$. It is supposed that every  $P_\theta\ll\mu$ and has a hereditary density $h_\theta$ with respect to $\mu$. Also the model is supposed to be identifiable. The associated Papangelou conditional intensities are denoted $\lambda_\theta$. For every pair $(\theta,\theta^*)$, consider the measure $\alpha_{\theta,\theta^*}$ on $D\times\mathscr{D}$ defined by
\begin{equation}
  \label{eq:def:alpha}
  \alpha_{\theta,\theta^*}(dx,dX) = \lambda_\theta(x,X)\;dxP_{\theta^*}(dX).
\end{equation}

The Kullback-Leibler divergence of $\alpha_{\theta,\theta^*}$ from $\alpha_{\theta^*,\theta^*}$ can be expressed as
\begin{equation*}
  \divkl\left(\alpha_{\theta^*,\theta^*},\alpha_{\theta,\theta^*}\right) =
  M(\theta^*,\theta^*)-M(\theta,\theta^*),
\end{equation*}
where
\begin{equation}
  \label{eq:contrast:campbell:kl:pp}
  M(\theta,\theta^*) =
  \int_{D\times\mathscr{D}} \log\lambda_\theta(x,X)\; C^!_{\theta^*}(dx,dX) -
  \int_{D\times\mathscr{D}} \lambda_\theta(x,X)\;dxP_{\theta^*}(dX)
\end{equation}
Note that $M(.,\theta^*)$ has a global maximum at $\theta=\theta^*$. Similarly to random T-tessellations, it can be shown that the maximum is strict.

Let $X$ be a realization of $\mathbf{X}\sim P_{\theta^*}$, then
\begin{equation}
  \label{eq:lpl:pp}
  \lpl(\theta|X) =
  \sum_{x\in X}  \log\lambda_\theta(x,X\setminus\{x\}) -
  \int_D \lambda_\theta(x,X)\;dx
\end{equation}
is an unbiased estimator of $M(\theta,\theta^*)$. Equation \eqref{eq:lpl:pp} coincides with the log-pseudolikelihood introduced by Besag \cite{ref/1243, ref/419, ref/1100} for point processes. The maximum pseudolikelihood estimator is
\begin{equation*}
  \hat{\theta} = \amax_\theta \lpl(\theta|X).
\end{equation*}

Generally no analytical expression of the integral involved in the log-pseudolikelihood \eqref{eq:lpl:pp} is available. Thus computing $\hat{\theta}$ is not straightforward. The most commly used method for computing $\hat{\theta}$ is the Poisson regression method. This method was first proposed by  Berman and Turner \cite{ref/1244} for the inference of heterogeneous Poisson point processes. The extension to a larger class of Gibbs processes was developped by Baddeley and Turner \cite{ref/758}. 
\subsection{A weaker but more tractable contrast function}
\label{sec-3-3}

\label{sec:pp:logistic:regression}
Below we introduce a family of contrast functions. They can be considered as approximations of the contrast function $-M$ associated with the pseudolikelihood. Their main feature is that their minimum is easily computed using standard tools of statistics. 

Let $\rho$ be an integrable function on $D$. In practice, $\rho$ may be taken proportional to the indicator function of $D$. Define
\begin{equation}
  \label{eq:def:alpha:rho}
  \alpha^\rho_{\theta,\theta^*}(dx, dX) = \left(
    \lambda_\theta(x,X)+\rho(x)
  \right)\;
  dx P_{\theta^*}(dX).
\end{equation}
When $\rho\equiv 0$, the measure defined in \eqref{eq:def:alpha:rho} is just $\alpha_{\theta,\theta^*}$ as defined in Equation \eqref{eq:def:alpha}.
Consider the function
\begin{equation}
  \label{eq:magic:contrast:pp}
  \divkl\left(\alpha_{\theta^*,\theta^*},\alpha_{\theta,\theta^*}\right) -
  \divkl\left(
    \alpha^\rho_{\theta^*,\theta^*},\alpha^\rho_{\theta,\theta^*}
  \right).
\end{equation}
This function cancels at $\theta=\theta^*$ and it is non-negative according to Lemma \ref{lemma:divergence:shift}. The following inequalities hold:
\begin{equation*}
  0 \leq 
  \divkl\left(\alpha_{\theta^*,\theta^*},\alpha_{\theta,\theta^*}\right) -
  \divkl\left(
    \alpha^\rho_{\theta^*,\theta^*},\alpha^\rho_{\theta,\theta^*}
  \right)
  \leq 
  \divkl\left(\alpha_{\theta^*,\theta^*},\alpha_{\theta,\theta^*}\right).
\end{equation*}
Heuristically, the function \eqref{eq:magic:contrast:pp} tends to the upper bound when  $\rho$ tends to infinity while it tends to the lower bound when $\rho$ tends to zero. Hence the function \eqref{eq:magic:contrast:pp} of $\theta$
\begin{itemize}
\item Approximates $\divkl\left(\alpha_{\theta^*,\theta^*},\alpha_{\theta,\theta^*}\right)$ as $\rho$ tends to infinity.
\item Is less discriminative than $\divkl\left(\alpha_{\theta^*,\theta^*},\alpha_{\theta,\theta^*}\right)$ for a finite $\rho$.
\end{itemize}

Simple computations show that up to some additive terms not depending on $\theta$ the KL-divergence can be written as
\begin{multline*}
  \divkl\left(
    \alpha^\rho_{\theta^*,\theta^*},\alpha^\rho_{\theta,\theta^*}
  \right) = -
  \int_{D\times\mathscr{D}} \log\left(\lambda_\theta(u,X)+\rho(u)\right)\;
  C_{\theta^*}^!(du,dX)\\ -
  \int_D \int_{\mathscr{D}}  \rho(u)    
  \log\left(\lambda_\theta(u,X)+\rho(u)\right)\; P_{\theta^*}(dX)du +
  \int_D\int_{\mathscr{D}} \lambda_\theta(u,X)\;P_{\theta^*}(dX)du\\+\ldots
\end{multline*}
Thus the difference \eqref{eq:magic:contrast:pp} can be written as
\begin{multline*}
  -\int_{D\times\mathscr{D}} \log\frac{
    \lambda_\theta(u,X)
  }{
    \lambda_\theta(u,X)+\rho(u)
  }\; C^!_{\theta^*}(du,dX) +
  \int_D \int_{\mathscr{D}}  \rho(u)    
  \log\left(\lambda_\theta(u,X)+\rho(u)\right)\; P_{\theta^*}(dX)du\\+\ldots
\end{multline*}
As a consequence the $M$ function introduced in \eqref{eq:contrast:campbell:kl:pp} can be redefined as follows
\begin{multline}
  \label{eq:contrast:campbell:kl:logistic:pp}
  M^\rho(\theta,\theta^*) =  
  \int_{D\times\mathscr{D}} \log\frac{
    \lambda_\theta(u,X)
  }{
    \left(\lambda_\theta(u,X)+\rho(u)\right)
  }\;
  C_{\theta^*}^!(du,dX)\\ -
  \int_D\int_{\mathscr{D}} 
    \rho(u) \log\left(\lambda_\theta(u,X)+\rho(u)\right)
  \;P_{\theta^*}(dX)du.
\end{multline}
The function $M^\rho$ has a global maximum $\theta=\theta^*$.
Below it is rewritten as an expectation in order to derive an unbiased estimator. The first term can be rewritten as
\begin{equation*}
  \int_{D\times\mathscr{D}} \log\frac{
    \lambda_\theta(u,X)
  }{
    \left(\lambda_\theta(u,X)+\rho(u)\right)
  }\;
  C_{\theta^*}^!(du,dX) = 
  \mean_{\theta^*} \sum_{x\in \mathbf{X}}
  \log\frac{
    \lambda_\theta(x,\mathbf{X}\setminus\{x\})
  }{
    \lambda_\theta(x,\mathbf{X}\setminus\{x\})+\rho(x)
  },
\end{equation*}
where $\mean_{\theta^*}$ is the mean when $\mathbf{X}\sim P_{\theta^*}$.
In order to rewrite the second term, consider a point process $\mathbf{Y}$ on $D$ with intensity $\rho$. Then we have
\begin{equation} 
  \label{eq:id:contrast:int2}
  \int_D\int_{\mathscr{D}} 
    \rho(u) \log\left(\lambda_\theta(u,X)+\rho(u)\right)
  \;P_{\theta^*}(dX)du  = 
  \mean \mean_{\theta^*} \sum_{y\in \mathbf{Y}}
  \log\left(\lambda_\theta(y,\mathbf{X})+\rho(y)\right),
\end{equation}
where the first mean is taken with respect to $\mathbf{Y}$.
Therefore an unbiased estimator of  $M^\rho$ is
\begin{equation}
  \label{eq:pre:lrl:pp}
  \sum_{x\in \mathbf{X}}
  \log\frac{
    \lambda_\theta(x,\mathbf{X}\setminus\{x\})
  }{
    \lambda_\theta(x,\mathbf{X}\setminus\{x\})+\rho(x)
  } -
  \sum_{y\in \mathbf{Y}}  \log\left(\lambda_\theta(y,\mathbf{X})+\rho(y)\right).
\end{equation}
The unbiasedness of \eqref{eq:pre:lrl:pp} can be checked as follows:
\begin{itemize}
\item The mean with respect to $\mathbf{X}\sim P_{\theta^*}$ of the first sum is equal to the first integral in \eqref{eq:contrast:campbell:kl:logistic:pp} by definition of the reduced Campbell measure \eqref{eq:def:campbell}.
\item The mean with respect to both $\mathbf{Y}$ and $\mathbf{X}\sim P_{\theta^*}$ of the second sum is equal to the second integral in \eqref{eq:contrast:campbell:kl:logistic:pp} in view of identity \eqref{eq:id:contrast:int2}.
\end{itemize}

Let $\hat{\theta}$ be the estimator of $\theta^*$ that maximizes the above criterion with respect to $\theta$. It maximizes also
\begin{equation}
  \label{eq:lrl:pp}
  \lpl^\rho(\theta|X) = 
  \sum_{x\in X}
  \log\frac{
    \lambda_\theta(x,X\setminus\{x\})
  }{
    \lambda_\theta(x,X\setminus\{x\})+\rho(x)
  } +
  \sum_{y\in Y}  \log\frac{
    \rho(y)
  }{
    \lambda_\theta(y,X)+\rho(y)
  }
\end{equation}
since both criteria differ only by a term not depending on $\theta$. In order to obtain an even more homogeneous expression, assume that the conditional probability given $\mathbf{X}$ that $\mathbf{X}\cap \mathbf{Y}=\emptyset$ is zero. Then almost surely $\mathbf{X}=\mathbf{X}\setminus\{y\}$ and the estimating criterion can be rewritten as
\begin{equation}
  \label{eq:lrl:pp:final}
  \lpl^\rho(\theta|X) = 
  \sum_{x\in X}
  \log\frac{
    \lambda_\theta(x,X\setminus\{x\})
  }{
    \lambda_\theta(x,X\setminus\{x\})+\rho(x)
  } +
  \sum_{y\in Y}  \log\frac{
    \rho(y)
  }{
    \lambda_\theta(y,X\setminus\{y\})+\rho(y).
  }
\end{equation}
The estimation criterion above was first proposed by Baddeley et al. \cite{ref/1247}. Our derivation here is different: based on a contrast function rather than a score function.

Thus we get the log-likelihood of a binomial regression. The sample to be considered is  $t_z,z\in X\cup Y$ with $t_z=1$ when $z\in X$. And the success probabilities are of the form
\begin{equation*}
  \frac{
    \lambda_\theta(z,X\setminus\{z\})
  }{
    \lambda_\theta(z,X\setminus\{z\})+\rho(z)
  }.
\end{equation*}
  
Let us focus on the case of an exponential family. Then we have
\begin{eqnarray*}
  \prob\left[t_z=1\right] & = &
  \frac{
    \exp\left(\theta^Tt(z,X\setminus\{z\})\right)
  }{
    \exp\left(\theta^Tt(z,X\setminus\{z\})\right)+\rho(z)
  },\\
  & = &
  \frac{1}{1+\exp\left(\log\rho(z)-\theta^Tt(z,X\setminus\{z\}\right)}.
\end{eqnarray*}
That is
\begin{equation*}
  \logit\prob\left[t_z=1\right] =
  \log\rho(z)-\theta^Tt(z,X\setminus\{z\}).
\end{equation*}
Therefore in order to compute $\hat{\theta}$ one can use computing tools already available for the binomial regression with a logistic link function.
\section{Conclusion}
\label{sec-4}

In our paper \cite{ref/1242}, a new class of stochastic models for planar tessellations was introduced. Theoretical concepts and results mainly inspired by point process theory were provided. Also a Metropolis-Hastings-Green type algorithm was proposed for simulating Gibbs models. The current report is a first step towards parametric inference from observed data. 

The key ingredient for introducing the pseudolikelihood for Gibbsian T-tessellations is the Kullback-Leibler divergence and contrast function theory. Applying such an approach to the density $h$ would lead to the classical likelihood which is unpracticable because of the unknown normalizing function involved in $h$. The trick here is to consider Campbell measures instead. As shown in this report,  applying a similar approach to point process theory yields the well-known pseudolikelihood introduced by Besag \cite{ref/1243} who discretized point processes into spatial random fields with discrete supports. It was also shown that the approximation of the maximum pseudolikelihood based on a binomial regression \cite{ref/1247} can be derived from a contrast function involving the Kullback-Leibler divergence and the Campbell measure. 

Whether such an approach can be extended to other types of random spatial structures is a open question. In our opinion, the key point is to identify basic operators that allow to explore the whole space of spatial structures. Based on such operators, one may define one or several Campbell measures leading in turn to a pseudolikelihood.

Alternative inference methods are also of interest. As a variant, one could consider two contrast functions associated with splits and merges on one side and on flips on the other side. This would lead to two pseudolikelihoods and two estimators. Another possibility is Monte-Carlo Maximum Likelihood (MCML), see e.g. \cite{ref/1178}. Implementing MCML is possible using the simulation algorithm proposed in \cite{ref/1242}. Bayesian approaches would deserve some investigations too.

In this report, the behaviour of the maximum pseudolikelihood estimator was investigated by means of simulations showing that bias and dispersion decrease as the number of cells in the observed tessellation increases. Theoretical studies would be of interest too, in particular asymptotics. However such studies require the development of a theoretical framework where random T-tessellation would be extended to the whole plane (currently random T-tessellations are bounded to a compact polygon).

The intriguing behaviour of the maximum pseudolikelihood when fitting a wrong model to observed data should be investigated either theoretically or empirically.

Applying the maximum pseudolikelihood to real data may require some preprocessing. As an example, consider the maximum pseudolikelihood for the CRTT model. Its computation requires the determination of the number of non-blocking segments and of the total cell perimeter. In practice, a T-tessellation is often provided as a list of polygons (cells). Therefore one has to compute the number of non-blocking segments from the list of polygons. This is not a trivial task both because real tessellations may not be exact T-tessellations and because of unavoidable numerical errors related to number representations in digital computing. The library LiTe provides some methods which may help for tranforming sets of polygons into T-tessellations.

\bibliographystyle{abbrv}
\bibliography{pseudo}
\appendix
\section{Extension of the Kullback-Leibler divergence to non probability measures}
\label{sec-5}

\label{sec:kl}
Let $\alpha$ and $\beta$ be two finite non-negative measures on a measurable space $\mathcal{E}$. Let the extended Kullback-Leibler divergence of $\beta$ from $\alpha$ be defined as
\begin{equation}
  \label{eq:def:KL}
  \divkl(\alpha,\beta) =
  \left\{
  \begin{array}{ll}
    \displaystyle\int_{\mathcal{E}} 
    \left(
      \frac{d\alpha}{d\beta}(z)\log\frac{d\alpha}{d\beta}(z) +
       1- \frac{d\alpha}{d\beta}(z)
    \right)
    \;\beta(dz), & \text{if }\beta\gg\alpha,\\
    &\\
    +\infty & \text{otherwise}.
  \end{array}
  \right.
\end{equation}
In the equation above, $0\log 0$ is assumed to be $0$. If $\alpha$ and $\beta$ have the same total mass, Equation \eqref{eq:def:KL} simplifies into
\begin{eqnarray*}
  \divkl(\alpha,\beta) 
   & = & \int_{\mathcal{E}} 
  \frac{d\alpha}{d\beta}(z)\log\frac{d\alpha}{d\beta}(z)\; \beta(dz).
\end{eqnarray*}
The simplification above applies especially when $\alpha$ and $\beta$ are probability measures. The classical Kullback-Leibler divergence is a particular case of the extended one.

Since the function
  \begin{equation*}
  x\rightarrow x\log x+1-x
\end{equation*}
is non-negative, the extended Kullback-Leibler divergence defined in Equation \eqref{eq:def:KL} is also non-negative. Furthermore the function above cancels only at $x=1$. It follows that the extended Kullback-Leibler divergence cancels only when $\alpha$ equals $\beta$ except on a $\beta$-null subset of $\mathcal{E}$. 
   
A small lemma that will be used further in this report.
\begin{lemma}
  \label{lemma:divergence:shift}
  Let $\alpha$, $\beta$ and $\gamma$ be non-negative finite measures on a measurable space $\mathcal{E}$. Then the following inequality holds
  \begin{equation*}
    \divkl(\alpha,\beta)\geq \divkl(\alpha+\gamma,\beta+\gamma).
  \end{equation*}
\end{lemma} 
The result is rather intuitive: $\alpha+\gamma$ and $\beta+\gamma$ are more similar than $\alpha$ and $\beta$. 
\begin{proof}
  First consider the case where $\beta$ does not dominate $\alpha$. Then $\divkl(\alpha,\beta)=+\infty$ and the inequality holds.

  Now assume that $\beta\gg\alpha$. Let $E$ be a measurable subset $E\subset\mathcal{E}$ such that $(\beta+\gamma)(E)=0$. Both $\beta$ and $\gamma$ cancel on $E$. Since $\beta\gg\alpha$, $\alpha$ cancels on $E$ too. It follows that $(\alpha+\gamma)(E)=0$. Hence $(\beta+\gamma)(E)=0$ implies that $(\alpha+\gamma)(E)=0$, i.e. $\beta+\gamma\gg\alpha+\gamma$.

  Let $\mu$ be a measure dominating $\alpha$, $beta$ and $\gamma$. Let $a$, $b$ and $c$ be the densities of $\alpha$, $\beta$ and $\gamma$ with respect to $\mu$. We have
  \begin{eqnarray*}
    \lefteqn{\divkl(\alpha+\gamma,\beta+\gamma)}\\ & = &
    \int_{\mathcal{E}} 
    \left(\frac{a(z)+c(z)}{b(z)+c(z)}
    \log\frac{a(z)+c(z)}{b(z)+c(z)} +
    1- \frac{a(z)+c(z)}{b(z)+c(z)}\right)
    \left(b(z)+c(z)\right)\;\mu(dz),\\
    & = &
    \int_{\mathcal{E}} 
    \left((a(z)+c(z))
    \log\frac{a(z)+c(z)}{b(z)+c(z)} +
    b(z)-a(z)\right)\;\mu(dz).
  \end{eqnarray*} 
  The difference between divergences can be written as
  \begin{eqnarray*}
    \lefteqn{
      \divkl(\alpha,\beta)-\divkl(\alpha+\gamma,\beta+\gamma)
    }\\ & = &
    \int_{\mathcal{E}} 
    \left(
      a(z) \log\frac{a(z)}{b(z)} -
      (a(z)+c(z)) \log\frac{a(z)+c(z)}{b(z)+c(z)}
    \right) \;d\mu(z).
  \end{eqnarray*}
  We have to check that the integrand is non-negative. The non-negativity results from the following inequality
  \begin{equation*}
    x\log\frac{x}{y} \geq (x+h)\log\frac{x+h}{y+h},\quad\text{for all }
    h\geq 0.
  \end{equation*}
  The latter inequality is a consequence of the decreasing of the function
  \begin{equation*}
    \begin{pmatrix} x\\y \end{pmatrix}
    \rightarrow x\log\frac{x}{y}
  \end{equation*}
  when moving away from the diagonal $y=x$. The decreasing is proved by
  \begin{equation*}
    \sprod{
      \nabla x\log\frac{x}{y}
    }{
      \begin{pmatrix}1\\1\end{pmatrix}
    } =
    \sprod{
      \begin{pmatrix}
        \log\frac{x}{y}+1\\ -\frac{x}{y}
      \end{pmatrix}
    }{
      \begin{pmatrix}1\\1\end{pmatrix}
    } =
    1-\frac{x}{y}+\log\frac{x}{y}\leq 0.
  \end{equation*}
  Above $\sprod{.}{.}$ denotes the scalar product.
\end{proof}
\section{Subconfigurations of a Poisson point process}
\label{sec-6}

\label{sec:subconfig:poisson}
\begin{lemma}\label{lem:subconfig:poisson}
  Let $\mathbf{X}$ be a Poisson point process with intensity measure $\nu$ on a space $D$. The total mass $\nu(D)$ is supposed to be finite. Then, for any measurable real-valued function $\phi$ defined on the space of finite point patterns in $D$, we have
  \begin{equation}
    \label{eq:mean:subconfig:poisson}
    \mean \sum_{Y\subset \mathbf{X}}\phi(Y) =
    \exp(\nu(D))\mean \phi(\mathbf{X}).
  \end{equation}
\end{lemma}
\begin{proof}
  Using the general expression of the distribution of a Poisson point process, one gets
  \begin{equation*}
    \mean\sum_{Y\in\mathbf{X}}\phi(Y)  = 
    e^{-\nu(D)}\sum_{n\geq 0}\frac{1}{n!}
    \int_{D^n}\sum_{Y\subset\{x_1,\ldots,x_n\}}\phi(Y)\;
    \nu^n(dx_1\ldots dx_n).
  \end{equation*}
  A subset of a point pattern can be mapped into a subset of indices:
  \begin{equation*}
    \mean\sum_{Y\in\mathbf{X}}\phi(Y)  = 
    e^{-\nu(D)}\sum_{n\geq 0}\frac{1}{n!} \sum_{I\subset\{1,\ldots,n\}}
    \int_{D^n}\phi(\{x_i:i\in I\})\;
    \nu^n(dx_1\ldots dx_n).
  \end{equation*}
  Grouping subsets by sizes yields
  \begin{multline*}
    \mean\sum_{Y\in\mathbf{X}}\phi(Y) =  
    e^{-\nu(D)}\sum_{n\geq 0}\frac{1}{n!} \sum_{m=0}^n 
    \sum_{\substack{I\subset\{1,\ldots,n\}\\|I|=m}} \nu(D)^{n-m}\\
    \int_{D^m}\phi(\{x_1,\ldots,x_m\})\;
    \nu^m(dx_1\ldots dx_m).
  \end{multline*}
  The summand does not depend on $I$. Therefore the sum over $I$ is just equal to the number of subsets of $m$ elements from a $n$-size set:
  \begin{eqnarray*}
    \lefteqn{\mean\sum_{Y\in\mathbf{X}}\phi(Y)}\\
    & = & e^{-\nu(D)}\sum_{n\geq 0}\frac{1}{n!} \sum_{m=0}^n 
    \frac{n!\nu(D)^{n-m}}{m!(n-m)!} 
    \int_{D^m}\phi(\{x_1,\ldots,x_m\})\;
    \nu^m(dx_1\ldots dx_m),\\
    & = & e^{-\nu(D)}\sum_{m\geq 0}\frac{e^{\nu(D)}}{m!} 
    \int_{D^m}\phi(\{x_1,\ldots,x_m\})\;
    \nu^m(dx_1\ldots dx_m),\\
    & = & e^{\nu(D)} \mean\phi(\mathbf{X}).
  \end{eqnarray*}
\end{proof}
\section{A random T-tessellation is determined by its Papangelou conditional intensities}
\label{sec-7}

\label{sec:charac:papangelou}
Below a proof of Proposition \ref{prop:sf:papangelou:determines:dist} is given. From equations (\ref{eq:split:papangelou}--\ref{eq:flip:papangelou}), it follows that
\begin{equation*}
  \left\{
    \begin{array}{rcl}
      h(sT) & = & \lambda_{\text{s}}(s,T)h(T),\\
      h(fT) & = & \lambda_{\text{f}}(f,T)h(T),
    \end{array}
  \right.
\end{equation*}
for $\mu$-almost all $T\in\mathcal{T}$ and almost all $s\in\mathbb{S}_T$ and all $f\in\mathbb{F}_T$. That is there exists a null set $E_{\text{s}}\subset\mathscr{C}_{\text{s}}$ with respect to the measure $ds\mu(dT)$ such that $h(sT)=\lambda_{\text{s}}(s,T)h(T)$ for all $(s,T)\notin E_\text{s}$. Similarly there exists a null set $E_{\text{f}}\subset\mathscr{C}_\text{f}$ with respect to the measure $df\mu(dT)$ such that $h(fT)=\lambda_{\text{f}}(f,T)h(T)$ for all $(f,T)\notin E_\text{f}$.

Let
\begin{equation*}
  G = \left\{
    T\in\mathcal{T}: \exists\tilde{T}\in\bigcup_{\tilde{L}\subset L(T)}\mathcal{T}(\tilde{L}), \exists o\in\mathbb{M}_{\tilde{T}}\cup\mathbb{F}_{\tilde{T}}, (o^{-1},o\tilde{T})\in E_{\text{s}}\cup E_{\text{f}} \right\},
\end{equation*}
where $L(T)$ is the pattern of lines supporting the segments of $T$.
The union
\begin{equation*}
  \bigcup_{\tilde{L}\subset L(T)}\mathcal{T}(\tilde{L})
\end{equation*}
is the subset of T-tessellations that is spanned by applying successive merges and flips to $T$ (remember that a merge remove a line supporting a segment and that a flip does not change the supporting lines).

Let $T\in\mathcal{T}$. There exists a sequence of splits and flips $o_n,\ldots,o_1$ which transform gradually the empty tessellation $D$ into $T$: $D=T_n,o_nT_n=T_{n-1},\ldots,o_2T_2=T_1,o_1T_1=T_0$ with $T_0=T$. Note that for $i=1,\ldots,n$, $o_i\in\mathbb{S}_{T_i}\cup\mathbb{F}_{T_i}$ i.e. $o_i^{-1}\in\mathbb{M}_{T_{i-1}}\cup\mathbb{F}_{T_{i-1}}$. Suppose that $T\notin G$.
\begin{equation*}
  T_0\notin G\Leftrightarrow \forall\tilde{T}\in\bigcup_{\tilde{L}\subset L(T)}\mathcal{T}(\tilde{L}),\forall o\in\mathbb{M}_{\tilde{T}}\cup\mathbb{F}_{\tilde{T}}, (o^{-1},o\tilde{T})\notin E_{\text{s}}\cup E_{\text{f}}.
\end{equation*}
The right-hand side statement holds for $\tilde{T}=T_0$. It follows 
\begin{equation*}
  T_0\notin G\Rightarrow \forall o\in\mathbb{M}_{T_0}\cup\mathbb{F}_{T_0}, (o^{-1},oT_0)\notin E_{\text{s}}\cup E_{\text{f}}.
\end{equation*}
The right-hand side statement holds for $o=o_1$:
\begin{eqnarray*}
  T_0\notin G & \Rightarrow & (o_1^{-1},o_1T_0)\notin E_{\text{s}}\cup E_{\text{f}},\\
& \Rightarrow & h(T_0) = \lambda_{t_1}(o_1,o^{-1}_1T_0)h(o_1^{-1}T_0),\\
& \Rightarrow & h(T_0) = \lambda_{t_1}(o_1,T_1)h(T_1),
\end{eqnarray*}
where for $i=1,\ldots,n$
\begin{equation*}
  t_i = \left\{
    \begin{array}{rl}
      \text{s} & \text{if }o_i\text{ is a split},\\
      \text{f} & \text{if }o_i\text{ is a flip}.
    \end{array}
  \right.
\end{equation*}
Repeating this computation, one gets
\begin{equation}
  \label{eq:ttessel:density:from:papangelou}
  h(T) = h(D)\prod_{i=1}^n \lambda_{t_i}(o_i,T_i).
\end{equation}

Now we show that $G$ is a $\mu$-null set. Let
\begin{eqnarray*}
  G_{\text{s}} & = & \left\{ T\in\mathcal{T} : \exists\tilde{T}\in\bigcup_{\tilde{L}\subset L(T)}\mathcal{T}(\tilde{L}), \exists o\in\mathbb{M}_{\tilde{T}}, (o^{-1},o\tilde{T})\in E_{\text{s}} \right\},\\
  G_{\text{f}} & = & \left\{
    T\in\mathcal{T}: \exists\tilde{T}\in\bigcup_{\tilde{L}\subset L(T)}\mathcal{T}(\tilde{L}), \exists o\in\mathbb{F}_{\tilde{T}}, (o^{-1},o\tilde{T})\in E_{\text{f}} \right\}.
\end{eqnarray*}
Obviously $G=G_{\text{s}}\cup G_{\text{f}}$. Let
\begin{equation*}
  F_{\text{s}} = \left\{
    \tilde{T}\in\mathcal{T}:
    \exists m\in\mathbb{M}_T, (m^{-1},m\tilde{T})\in E_{\text{s}}
  \right\}.
\end{equation*}
We have
\begin{equation*}
  T\in G_{\text{s}} \Leftrightarrow
  \exists \tilde{T}\in\bigcup_{\tilde{L}\subset L(T)}\mathcal{T}(\tilde{L}),
  \tilde{T}\in F_{\text{s}}.
\end{equation*}
Therefore
\begin{equation*}
  \mu(G_{\text{s}})  \leq  \mean \sum_{\tilde{T}\in\bigcup_{\tilde{L}\subset L(\mathbf{T})}\mathcal{T}(\tilde{L})} I_{F_{\text{s}}}(\tilde{T}),\quad \mathbf{T}\sim\mu
\end{equation*}
Using Definition \ref{def:crtt}, one can rewrite the right-hand side as a mean with respect to the unit Poisson line process $\mathbf{L}$:
\begin{eqnarray*}
  \mu(G_{\text{s}}) & \leq & \text{cte} \mean \sum_{T\in\mathcal{T}(\mathbf{L})}
  \sum_{\tilde{T}\in\bigcup_{\tilde{L}\subset L(T)}\mathcal{T}(\tilde{L})} I_{F_{\text{s}}}(\tilde{T}),\\
  & \leq & \text{cte} \mean \sum_{T\in\mathcal{T}(\mathbf{L})}
  \sum_{\tilde{T}\in\bigcup_{\tilde{L}\subset \mathbf{L}}\mathcal{T}(\tilde{L})} I_{F_{\text{s}}}(\tilde{T}),\\
  & \leq & \text{cte} \mean |\mathcal{T}(\mathbf{L})| 
  \sum_{\tilde{T}\in\bigcup_{\tilde{L}\subset \mathbf{L}}\mathcal{T}(\tilde{L})} I_{F_{\text{s}}}(\tilde{T}).
\end{eqnarray*}
The constant before the mean is the unknown normalizing constant of the distribution $\mu$. Now let us show that the random variable
\begin{equation*}
  \sum_{\tilde{T}\in\bigcup_{\tilde{L}\subset \mathbf{L}}\mathcal{T}(\tilde{L})} I_{F_{\text{s}}}(\tilde{T})
\end{equation*}
is null almost surely. Since it is non-negative, it is sufficient to show that its mean is null. Since the union is disjoint, we have
\begin{equation*}
  \mean \sum_{\tilde{T}\in\bigcup_{\tilde{L}\subset \mathbf{L}}\mathcal{T}(\tilde{L})} I_{F_{\text{s}}}(\tilde{T})  =  
  \mean \sum_{\tilde{L}\subset\mathbf{L}} 
  \sum_{\tilde{T}\in\mathcal{T}(\tilde{L})} I_{F_{\text{s}}}(\tilde{T}).
\end{equation*}
Applying Lemma \ref{lem:subconfig:poisson} (see Appendix \ref{sec:subconfig:poisson}), we can rewrite the mean of the sum on the line subsets:
\begin{equation*}
  \mean \sum_{\tilde{T}\in\bigcup_{\tilde{L}\subset \mathbf{L}}\mathcal{T}(\tilde{L})} I_{F_{\text{s}}}(\tilde{T})  =  \exp(b(D)/\pi)
  \mean \sum_{\tilde{T}\in\mathcal{T}(\mathbf{L})} I_{F_{\text{s}}}(\tilde{T}),
\end{equation*}
where $b(D)$ is the perimeter of the domain $D$.
Therefore it follows from Definition \ref{def:crtt}
\begin{equation*}
  \mean \sum_{\tilde{T}\in\bigcup_{\tilde{L}\subset \mathbf{L}}\mathcal{T}(\tilde{L})} I_{F_{\text{s}}}(\tilde{T})  =  \exp(b(D)/\pi) \mu(F_{\text{s}})
\end{equation*}
From the definition of $F_{\text{s}}$ we have
\begin{eqnarray*}
  \mu(F_{\text{s}}) & \leq & \mean \sum_{m\in\mathbb{M}_\mathbf{T}} 
  I_{E_{\text{s}}}(m^{-1},mT),\quad
  \mathbf{T}\sim\mu,\\
  & \leq & \int_{\mathcal{T}}\int_{\mathbb{S}_T} I_{E_{\text{s}}}(s,T)\;
  ds\mu(dT),\\
  & \leq & 0.
\end{eqnarray*}
Hence
\begin{equation*}
  \mean \sum_{\tilde{T}\in\bigcup_{\tilde{L}\subset \mathbf{L}}\mathcal{T}(\tilde{L})} I_{F_{\text{s}}}(\tilde{T})  = 0,
\end{equation*}
and
\begin{equation*}
  \sum_{\tilde{T}\in\bigcup_{\tilde{L}\subset \mathbf{L}}\mathcal{T}(\tilde{L})} I_{F_{\text{s}}}(\tilde{T})
\end{equation*}
is null almost surely. In turn, this implies that
\begin{equation*}
  \mean |\mathcal{T}(\mathbf{L})| 
  \sum_{\tilde{T}\in\bigcup_{\tilde{L}\subset \mathbf{L}}\mathcal{T}(\tilde{L})} I_{F_{\text{s}}}(\tilde{T}) = 0
\end{equation*}
and that $\mu(G_{\text{s}})=0$. The proof for $\mu(G_{\text{f}})=0$ is similar.

In Equation \eqref{eq:ttessel:density:from:papangelou}, the density $h$ is determined only up to the multiplicative constant $h(D)$. It should be noticed that $h(D)$ is necessarily positive (otherwise $h$ would be null on $\mathcal{T}$ due to its hereditariness). The constant $h(D)$ is determined from the fact that the integral of $h$ on $\mathcal{T}$ is equal to 1.

\end{document}